\documentclass[11pt]{amsart}

\usepackage{amscd, amssymb, amsthm, enumerate}
\usepackage[all]{xy}

\newtheorem{theorem}{Theorem}[section]
\newtheorem{corollary}[theorem]{Corollary}
\newtheorem{lemma}[theorem]{Lemma}
\newtheorem{proposition}[theorem]{Proposition}

\theoremstyle{definition}
\newtheorem{definition}[theorem]{Definition}
\theoremstyle{remark}
\newtheorem{remark}[theorem]{Remark}
\newtheorem{example}[theorem]{Example}

\newcommand{\cst}{$\rm{C}^*$}
\newcommand{\czero}{{\rm{C}}_0}
\newcommand{\ccomp}{{\rm{C}}_c}
\newcommand{\ccont}{{\rm{C}}}
\newcommand{\cf}{\textit{cf.~}}

\newcommand{\ie}{\textit{i.e.~}}
\newcommand{\mcl}[1]{\text{$\mathcal{#1}$}}
\newcommand{\newfunk}[3]{{#1}\colon{#2}\to {#3}}
\newcommand{\qqxz}{\QQ_n \times_{\delta} \ZZ}

\newcommand{\NN}{\mathbb{N}}

\newcommand{\ZZ}{\mathbb{Z}}
\newcommand{\BB}{\mathbb{B}}
\newcommand{\EE}{\mathbb{E}}
\newcommand{\FF}{\mathbb{F}}
\newcommand{\QQ}{\mathbb{Q}}
\newcommand{\RR}{\mathbb{R}}

\newcommand{\Ll}{\mathcal{L}}

\newcommand{\Oo}{\mathcal{O}}

\begin{document}

\title[Purely infinite partial crossed products]{Purely infinite partial crossed products}

\author[Thierry Giordano]{Thierry Giordano${}^1$}
\address{Department of Mathematics and Statistics, University of Ottawa, 585 King Edward Ave, K1N6P1 Ottawa, Canada}
\email{giordano@uottawa.ca}
\thanks{${}^1$Research supported by NSERC, Canada.}

\author[Adam Sierakowski]{Adam Sierakowski${}^2$}
\address{The School of Mathematics \& Applied Statistics,  University of Wollongong,  Northfields Ave, 2522 NSW, Australia}
\email{asierako@uow.edu.au}
\thanks{${}^2$Communicating author. Research supported by Professor Collins, Handelman and Giordano, Canada.}

\date{May 29 2013}
\maketitle
\begin{abstract}
Let $(\mcl{A},G,\alpha)$ be a partial dynamical system. We show that there is a bijective correspondence between $G$-invariant ideals of $\mcl{A}$ and ideals in the partial crossed product $\mcl{A}\rtimes_{\alpha,r} G$ provided the action is exact and {residually topologically} free. Assuming, in addition, a technical condition---automatic when $\mcl{A}$ is abelian---we show that $\mcl{A}\rtimes_{\alpha,r} G$ is purely infinite if and only if the positive nonzero elements in $\mcl{A}$ are properly infinite in $\mcl{A}\rtimes_{\alpha,r} G$. As an {application} we verify {pure} infiniteness of various partial crossed products, including realisations of the Cuntz algebras $\Oo_n$, $\Oo_A$, $\Oo_\NN$, and $\Oo_\ZZ$ as partial crossed products.
\end{abstract}

\section{Introduction}

In the theory of operator algebras, the crossed product construction has been one of the most important and fruitful tools both to construct examples and to describe the internal structure of operator algebras (in particular the von Neumann algebras).

Partial actions of a discrete group on \cst-algebras and their associated crossed products were gradually introduced in \cite{Exe2} and \cite{McClan}, and since then developed by many authors. Several important classes of \cst-algebras have been realised as crossed products by partial actions, including in particular AF-algebras, Cuntz-Krieger algebras, Bunce-Deddens algebras, among others (see for example \cite{BoaExe, ExeGioGon, Exe5, Exe6, Exe, ExeLacQui, Hop}). The description of \cst-algebras as partial crossed products has also proved useful for the computation of their K-theory.

\medskip

In this paper we pursue the study of partial \cst-dynamical systems and their crossed products associated. We begin by recalling (in Section \ref{section1} and Appendix \ref{appendixA1}) the construction of the partial crossed product $\mcl{A}\rtimes_{\alpha,r}G$ associated to a {\em partial action} $\alpha$ (a compatible collection of isomorphisms $\newfunk{\alpha_t}{\mcl{D}_{t^{-1}}}{\mcl{D}_t}$, $t\in G$ of ideals in $\mcl{A}$). 

\medskip

In Section \ref{section2}, we study the ideal structure of partial crossed products, generalising the results on \cst-dynamical systems obtained by the second author in \cite{Sie}. Recall that a partial action $\alpha$ on a \cst-algebra $\mcl{A}$ has the {\em the intersection property} if every nontrivial ideal in $\mcl{A}\rtimes_{\alpha,r}G$ intersects $\mcl{A}$ nontrivially. Then (Definition \ref{def2.1}) we say that a partial \cst-dynamical system $(A, G, \alpha)$ has the {\em residual intersection property} if for every $G$-invariant ideal $\mcl{I}$ in $\mcl{A}$, the induced partial action of $G$ on $\mcl{A}/\mcl{I}$ has the intersection property. We establish in Theorem \ref{mainthorem} a one-to-one correspondence between ideals in $\mcl{A}\rtimes_{\alpha,r}G$  and $G$-invariant ideals of $\mcl{A}$ provided---in fact if and only if---the partial action $\alpha$ is exact and has the residual intersection property. When the partial action $\alpha$ is {\em minimal} (no nontrivial $G$-invariant ideals in $\mcl{A}$), then the {\em exactness} (any nontrivial $G$-invariant ideal $\mcl{I}$ in $\mcl{A}$ induces a short exact sequence at the level of reduced crossed products) is automatic. 

In \cite{ExeLacQui}, having defined a topologically free partial action (by partial homeomorphisms) on a locally compact space $X$, Exel, Laca, and Quigg proved the simplicity of the partial crossed product $\czero(X)\rtimes_{\alpha,r}G$ under the presence of minimality and topological freeness of $\alpha$. 
In \cite{Leb}, Lebedev extended the definition of topological freeness to non-commutative partial actions and showed that a topologically free partial action has always {\it the intersection property}. With Lebedev's result, we recover in Corollary 2.9 the theorem of Exel, Laca, and Quigg. 

Theorem \ref{mainthorem} allows us also to extend Echterhoff and Laca's work on crossed products: We show that the canonical map $\mcl{J}\mapsto \mcl{J}\cap \mcl{A}$ between ideals in $\mcl{A}\rtimes_{\alpha,r}G$ and $G$-invariant ideals of $\mcl{A}$ restricts to a continuous map from the space of prime ideals of $\mcl{A}\rtimes_{\alpha,r}G$ to the space of $G$-prime ideals in $\mcl{A}$. This restriction is a homeomorphism provided $\alpha$ is exact and residually topologically free, where residual topological freeness is an ideal related version of topological freeness. When $\mcl{A}$ is separable and abelian we show that the space of prime ideals of $\mcl{A}\rtimes_{\alpha,r}G$ is homeomorphic to the quasi-orbit space of ${\rm Prim}\,\mcl{A}$. 

\medskip

In Section \ref{section3} we generalise some of the main results in \cite{RorSie}, by R{\o}rdam and the second named author, to partial \cst-dynamical systems. In particular we give sufficient conditions for a partial crossed products to be purely infinite in the sense of Kirchberg and R{\o}rdam. One of the keys assumptions of Theorem \ref{the4.1} goes 
back to Elliott's notion of {\em proper outerness} (an automorphism $\alpha$ of $\mcl{A}$ is properly outer if $\|\alpha|_{\mcl{I}}-\beta\|=2$ for every $\alpha$-invariant ideal $\mcl{I}$ in $\mcl{A}$ and any inner automorphism $\beta$ of $\mcl{I}$). Based on Kishimoto's work, Olsen and Pedersen proved that an automorphism $\alpha$ of a separable \cst-algebra $\mcl{A}$ is properly outer if and only if $\inf\{\|x (a \delta) x\|: x\in \mcl{B}_+, \|x\|=1\}=0$ for every $a\in{\mcl{A}}$, and every nonzero hereditary \cst-algebra $\mcl{B}$ in $\mcl{A}$, where $\delta$ is the unitary implementing $\alpha$, i.e, $\alpha(a)=\delta a \delta^*$ in $\mcl{A}\rtimes_{\alpha}\ZZ$. Moreover Archbold and Spielberg proved that topological freeness of an action $\alpha$ ensures that $\alpha_t$, $t\neq e$ is properly outer. In Proposition \ref{abeliantopfree}, generalising these results, we prove the equivalence of the following statements for a partial action $\alpha$ on an abelian \cst-algebra $\mcl{A}$: 
\begin{enumerate}[(i)]
\item $\alpha$ is topologically free,
\item $\|\alpha_t|_{\mcl{I}}-id|_\mcl{I}\|=2$ for every $\alpha_t$-invariant ideal $\mcl{I}$ in $\mcl{D}_{t^{-1}}$ (and $t\neq e$),
\item $\inf\{\|x( a\delta_t)x\|: x\in \mcl{B}_+, \|x\|=1\}=0$ for every $a\in{\mcl{D}_t}$, and every nonzero hereditary \cst-algebra $\mcl{B}$ in $\mcl{A}$ (and $t\neq e$).
\end{enumerate}
Without assuming that $\mcl{A}$ is commutative and under the assumption of property (iii) for actions of $G$ on $\mcl{A}/\mcl{I}$ for any $G$-invariant ideal $\mcl{I}$ in $\mcl{A}$ and exactness of $\alpha$ we once again obtain a one-to-one correspondence between ideals in $\mcl{A}\rtimes_{\alpha,r}G$ and $G$-invariant ideals in $\mcl{A}$. However, in this context we can state in Theorem \ref{the4.1} sufficient and necessary conditions for pure infiniteness of the partial crossed products: $\mcl{A}\rtimes_{\alpha,r}G$ is purely infinite if and only if every nonzero positive element in $\mcl{A}$ is properly infinite in
$\mcl{A}\rtimes_{\alpha,r} G$ (under presence of ideal property of $\mcl{A}$).

In the second part of Section \ref{section3}, we study pure infiniteness of partial crossed products $\mcl{A}\rtimes_{\alpha,r} G$ when $\mcl{A}$ is abelian. We establish a geometrical condition sufficient---and sometimes also necessary---to obtain pure infiniteness of partial crosses products. Specifically we show in Theorem \ref{the4.2}  that an exact and residually topologically free partial action on a totally disconnected locally compact Hausdorff space $X$ gives a purely infinite partial crossed product $\czero(X)\rtimes_{\alpha,r} G$ provided that every compact and open subset of $X$ is $(G,\tau_X)$-paradoxical (Definition \ref{paradox}).

\medskip

{In Section \ref{section4},  we apply our results to the {a variety of} know examples of partial crossed products:}

(i) Hopenwasser constructs a partial action of the semidirect product $\qqxz$ of $n$-adic rationals by the integers $\ZZ$ on the Cantor set $X_C$ such that the partial crossed product $\ccont(X_C)\rtimes_{\alpha} (\qqxz)$ is isomorphic to the Cuntz algebra $\mathcal{O}_n$. We verify that for this (exact and residually topologically free) action the compact and open subset of $X_C$ are $(\qqxz,\tau_{X_C})$-paradoxical.

(ii) For a $\{0,1\}$-valued $n$ by $n$ matrix $A=[a_{ij}]$ with no zero rows Exel, Laca, and Quigg realised $\mathcal O_A=\ccont^*(s_1,\dots, s_n: \sum_j s_j s_j^* = 1, \sum_j a_{ij} s_j s_j^* = s_i^* s_i)$ as a partial crossed product $\ccont(X_A)\rtimes_{\alpha} \FF_n$, where $\FF_n$ denotes the free group of rank $n$. The algebra $\mathcal O_A$ was defined by Astrid an Huef and Raeburn as an universal analogue of the Cuntz-Krieger algebra. When $\alpha$ is exact and residually topologically free we prove that $\ccont(X_A)\rtimes_{\alpha} \FF_n$ is purely infinite if and only if the compact and open subset of $X_A$ are $(\FF_n,\tau_{X_A})$-paradoxical. We use graph \cst-algebraic results by Raeburn, Kumjian, Pask, Raeburn, Renault, Mann, Sutherland, Hong and Szymanski among others.

(iii) Boava and Exel showed that for each integral domain $R$ with finite quotients $R/(m)$, $m\neq 0$, the semidirect product $K \rtimes K^\times$ of $K$ by $K\backslash\{0\}$ (where $K$ is the field of fractions of $R$) acts on the Cantor set $X_R$ such that the partial crossed product $\ccont(X_R)\rtimes_{\alpha} (K \rtimes K^\times)$ is isomorphic to the regular \cst-algebra $\mathfrak{A}[R]$ of $R$. We verify that for this (exact and residually topologically free) action the compact and open subset of $X_R$ are $(K \rtimes K^\times,\tau_{X_R})$-paradoxical.

\medskip
 
In \cite{Aba}, Abadie described a class of partial crossed products $\czero(X)\rtimes_{\alpha,r} G$ Morita-Rieffel equivalent to ordinary crossed products. In Section \ref{section5}, we prove that for such a partial crossed product $\czero(X)\rtimes_{\alpha,r} G$ if the compact and open subsets in the spectrum of the the original \cst-algebra of the partial crossed product are paradoxical, then the same property must holds for the corresponding (ordinary) crossed product.

Much of the work described in this paper was done while the second author held a postdoctoral fellowship at the University of Ottawa. He wishes to thank the members of the Department of Mathematics and Statistics at the University of Ottawa for their warm hospitality.

\section{Definition of partial crossed product.}
\label{section1}
Let $\mcl{A}$ be a \cst-algebra and $G$ be a discrete group. We recall the definition of a partial dynamical system and the corresponding partial crossed products (see Appendix \ref{appendixA1} for more details).

\begin{definition}
Let $\mcl{A}$ be a \cst-algebra and $G$ be a discrete group. A {\em partial action of $G$ on $\mcl{A}$}, denoted by $\alpha $, is a collection $(\mcl{D}_t)_{t \in G}$ of closed two-sided ideals of $\mcl{A}$ and a collection $(\alpha_t)_{t \in G}$ of $^*$-isomorphisms $\newfunk{\alpha_t}{\mcl{D}_{t^{-1}}}{\mcl{D}_t}$ such that

\begin{enumerate}[(i)]
\item $\mcl{D}_e=\mcl{A}$, where $e$ represents the identity element of $G$;
\item $\alpha_s^{-1}(\mcl{D}_s\cap \mcl{D}_{t^{-1}})\subseteq\mcl{D}_{(ts)^{-1}}$;
\item $\alpha_t\circ \alpha_s(x)=\alpha_{ts}(x), \ \ \forall \: x\in\alpha_s^{-1}(\mcl{D}_s\cap \mcl{D}_{t^{-1}})$.
\end{enumerate}
\end{definition}

The triple $(\mcl{A},G,\alpha)$ is called a {\em partial dynamical system}. The equivalence between this definition of partial action by Dokuchaev and Exel in \cite{DokExe} and the original definition of McClanahan (\cite{McClan}) was proven in  \cite{QuiRae}. For the case when $\mcl{A}$ is abelian we refer to \cite{Aba, ExeGioGon, ExeLacQui}.

\begin{definition}
\label{delpartialsystem}
Let $(\mcl{A},G,\alpha)$ be a partial dynamical system. Let $\mcl{L}$ be the normed $^*$-algebra of the finite formal sums $\sum_{t \in G} a_t\delta_t$, where $a_t\in\mcl{D}_t$. The operations and  the norm in $\mcl{L}$ are given by
$$(a_t\delta_t)(a_s\delta_s)=\alpha_t(\alpha_{t^{-1}}(a_t)a_s)\delta_{ts}, \ \ \ (a_t\delta_t)^*=\alpha_{t^{-1}}(a_t^*)\delta_{t^{-1}},$$ $$\|\sum_{t \in G} a_t\delta_t\|=\sum_{t \in G} \|a_t\|.$$
Let $B_t$ denote the vector subspace $\mcl{D}_t\delta_t$ of $\mcl{L}$. The family $(B_t)_{t \in G}$ generates a Fell bundle. The full crossed product $\mcl{A}\rtimes_{\alpha}G$ and the reduced crossed product $\mcl{A}\rtimes_{\alpha,r}G$ are, respectively, the full and the reduced cross sectional algebras of $(B_t)_{t \in G}$. Both crossed products are completions of $\mcl{L}$ with respect to a certain \cst-norm. We recall the construction of these crossed products in Appendix \ref{appendixA1}. 
\end{definition}
We suppress the canonical inclusion map $\mcl{A}\to \mcl{A}\rtimes_{\alpha}G$, $a\mapsto a\delta_e$ and view $\mcl{A}$ as a sub-\cst-algebra of $\mcl{A}\rtimes_{\alpha}G$. All ideals (throughout out this paper) are assumed to be closed and two-sided. The set $\tau_X$ denotes the topology of a topological space $X$. The $C^*$-algebra of continuous functions vanishing at infinity on a locally compact Hausdorff space $X$ is denoted by  $\czero(X)$. Every abelian $C^*$-algebra arises in this form. When the algebra is unital then $X$ is compact and we emphasise this fact by writing it as $\ccont(X)$.

\section{Ideal structure of partial crossed product.}
\label{section2}
Before stating our results on the ideal structure of a partial crossed product {generalising \cite{Sie}} we need a few definitions.
	
\begin{definition}
\label{def2.1}
Let $(\mcl{A},G,\alpha)$ be a partial dynamical system. Then
\begin{itemize}
\item[(i)] A closed two-sided ideal $\mcl{I}$ of $\mcl{A}$ is \emph{$G$-invariant} provided that $\alpha_t(\mcl{I}\cap \mcl{D}_{t^{-1}})\subseteq \mcl{I}$ for every $t\in G$. 
\item[(ii)] The partial action is \emph{exact} if every $G$-invariant ideal $\mcl{I}$ of $\mcl{A}$ induces a short exact sequence 
\begin{displaymath}
\xymatrix{0 \ar[r] & \mcl{I}\rtimes_{\alpha,r} G  \ar[r] & \mcl{A}\rtimes_{\alpha,r} G  \ar[r] & \mcl{A/I}\rtimes_{\alpha,r} G \ar[r] & 0}
\end{displaymath}
at the level of reduced crossed products. 
\item[(iii)] The partial action has the \emph{residual intersection property} if for every $G$-invariant ideal $\mcl{I}$ of $\mcl{A}$ the intersection of $\mcl{A/I}$ with any nonzero ideal in $\mcl{A/I}\rtimes_{\alpha,r} G$ is nonzero.
\end{itemize}
\end{definition}

\begin{theorem}
\label{mainthorem}
Let $(\mcl{A},G,\alpha)$ be a partial dynamical system. There is a one-to-one correspondence between ideals in $\mcl{A}\rtimes_{\alpha,r}G$  and $G$-invariant ideals of $\mcl{A}$ if and only if the partial action is exact and has the residual intersection property. 
\end{theorem}

\begin{proof}
\emph{Sufficiency:} Suppose that $\alpha$ is exact and has the residual intersection property. Let $\newfunk{E_{\mcl{A}}}{\mcl{A}\rtimes_{\alpha,r}G}{\mcl{A}}$ denote the usual conditional expectation on the crossed product (see Appendix \ref{appendixA1}). Let ${\rm Ideal}[S]$ denote the smallest ideal in $\mcl{A}\rtimes_{\alpha,r}G$ generated by $S\subseteq \mcl{A}\rtimes_{\alpha,r}G$. Let $\varphi$ denote the map $\mcl{J}\mapsto\mcl{J\cap A}$ from the ideals in $\mcl{A}\rtimes_{\alpha,r}G$ into $G$-invariant ideals in $\mcl{A}$. Using that ${\rm Ideal}[\mcl{I}]\cap \mcl{A} = \mcl{I}$, for any $G$-invariant ideals $\mcl{I}$ of $\mcl{A}$, \cf{}Proposition \ref{appendix3}, we conclude that $\varphi$ is surjective. To show $\varphi$ is injective it is enough to show that $\mcl{J}={\rm Ideal}[E_{\mcl{A}}(\mcl{J})]$ for every ideal $\mcl{J}$ in $\mcl{A}\rtimes_{\alpha,r}G$. If we have two ideals $\mcl{J}_1$, $\mcl{J}_2$ of $\mcl{A}\rtimes_{\alpha,r}G$ with the same intersection it then follows that
$$\mcl{J}_1= {\rm Ideal}[E_{\mcl{A}}(\mcl{J}_1)\cap\mcl{A}]\subseteq {\rm Ideal}[\mcl{J}_1\cap\mcl{A}]= {\rm Ideal}[\mcl{J}_2\cap\mcl{A}]\subseteq {\rm Ideal}[E_{\mcl{A}}(\mcl{J}_2)]=\mcl{J}_2,$$
using that $E_{\mcl{A}}(\mcl{J}_1)\subseteq {\rm Ideal}[E_{\mcl{A}}(\mcl{J}_1)]=\mcl{J}_1$, and $\mcl{J}_2\cap \mcl{A}\subseteq E_{\mcl{A}}(\mcl{J}_2)$ (\cf{}Proposition \ref{appendix3}). Fix any ideal $\mcl{J}$ of $\mcl{A}\rtimes_{\alpha,r}G$.

{(i) $\mcl{J}\subseteq {\rm Ideal}[E_{\mcl{A}}(\mcl{J})]$:} Let $\mcl{I}$ denote the (smallest) $G$-invariant ideal in $\mcl{A}$ generated by $E_{\mcl{A}}(\mcl{J})$. By Proposition \ref{appendix1} we have the commuting diagram
$$
\xymatrix{0 \ar[r] & \mcl{I}\rtimes_{\alpha,r}G \ar[r]^-{\iota} \ar[d]^-{E_\mcl{I}} & \mcl{A}\rtimes_{\alpha,r}G  \ar[r]^-{\rho} \ar[d]^-{E_\mcl{A}}& \mcl{A/I}\rtimes_{\alpha,r}G \ar[r] \ar[d]^-{E_{\mcl{A/I}}}& 0\\
0 \ar[r] & \mcl{I} \ar[r] & \mcl{A} \ar[r] & \mcl{A/I} \ar[r] & 0}
$$	
For each $x\in \mcl{J}^{+}$ we have $E_{\mcl{A}}(x)\in \mcl{I}$ and hence that $E_{\mcl{A/I}}(\rho(x))=E_{\mcl{A}}(x)+\mcl{I}=0$. By exactness $\ker \rho \subseteq {\rm Ideal}[\mcl{I}]$ (not true in general). Hence $x\in \ker \rho = {\rm Ideal}[\mcl{I}] = {\rm Ideal}[E_{\mcl{A}}(\mcl{J})]$. Since every element in $\mcl{J}$ is a linear combination of positive elements $\mcl{J}\subseteq {\rm Ideal}[E_{\mcl{A}}(\mcl{J})]$.

{(ii) ${\rm Ideal}[E_{\mcl{A}}(\mcl{J})]\subseteq \mcl{J}$:} Let $\mcl{I}$ denote the intersection $\mcl{J}\cap \mcl{A}$. By exactness we have {a} short exact sequence
$$
\xymatrix{0 \ar[r] & \mcl{I}\rtimes_{\alpha,r}G \ar[r]^-{\iota} & \mcl{A}\rtimes_{\alpha,r}G  \ar[r]^-{\rho} & \mcl{A/I}\rtimes_{\alpha,r}G \ar[r] & 0\\
}
$$
With $\mcl{K}:=\mcl{I}\rtimes_{\alpha,r}G \subseteq \mcl{J}$ we can therefore make the following identifications:
$$\rho(\mcl{J})=\mcl{J/K}, \ \ \ \mcl{A/I=A/(K\cap A)=(A+K)/K}.$$
Suppose that $\mcl{J/K}$ and $\mcl{(A+K)/K}$ has a nonzero intersection. Then there exist $j\in \mcl{J}$ and $a\in \mcl{A}$ such that $j+\mcl{K}=a+\mcl{K}\neq \mcl{K}$. Since $\mcl{K}\subseteq \mcl{J}$ it follows that $a\in \mcl{J}$ and hence that $a\in \mcl{J\cap A=I\subseteq K}$. But then $a+\mcl{K}=\mcl{K}$ giving a contradiction. We conclude that $\rho(\mcl{J})\cap \mcl{A/I}=0$. The residual intersection property implies that $\rho(\mcl{J})=0$. By Proposition \ref{appendix1} we have the commuting diagram
$$
\xymatrix{0 \ar[r] & \mcl{I}\rtimes_{\alpha,r}G \ar[r]^-{\iota} \ar[d]^-{E_\mcl{I}} & \mcl{A}\rtimes_{\alpha,r}G  \ar[r]^-{\rho} \ar[d]^-{E_\mcl{A}}& \mcl{A/I}\rtimes_{\alpha,r}G \ar[r] \ar[d]^-{E_{\mcl{A/I}}}& 0\\
0 \ar[r] & \mcl{I} \ar[r] & \mcl{A} \ar[r] & \mcl{A/I} \ar[r] & 0}
$$	
For each $x\in \mcl{J}$ we have that $E_{\mcl{A/I}}(\rho(x))=E_{\mcl{A}}(x)+\mcl{I}=0$. Hence $E_{\mcl{A}}(x)\in \mcl{I=J\cap A}$. We conclude that ${\rm Ideal}[E_{\mcl{A}}(\mcl{J})]\subseteq \mcl{J}$.

\emph{Necessity:} Suppose that $\varphi$ is bijective, where $\varphi$ denotes the map $\mcl{J}\mapsto\mcl{J\cap A}$ from the ideals in $\mcl{A}\rtimes_{\alpha,r}G$ into $G$-invariant ideals in $\mcl{A}$. As previously let $\newfunk{E_{\mcl{A}}}{\mcl{A}\rtimes_{\alpha,r}G}{\mcl{A}}$ denote the conditional expectation on the crossed product.

{(iii) Exactness:} Fix any $G$-invariant ideal $\mcl{I}$ of $\mcl{A}$. By Proposition \ref{appendix1} we have the commuting diagram
$$
\xymatrix{0 \ar[r] & \mcl{I}\rtimes_{\alpha,r}G \ar[r]^-{\iota} \ar[d]^-{E_\mcl{I}} & \mcl{A}\rtimes_{\alpha,r}G  \ar[r]^-{\rho} \ar[d]^-{E_\mcl{A}}& \mcl{A/I}\rtimes_{\alpha,r}G \ar[r] \ar[d]^-{E_{\mcl{A/I}}}& 0\\
0 \ar[r] & \mcl{I} \ar[r] & \mcl{A} \ar[r] & \mcl{A/I} \ar[r] & 0}
$$
By assumption $\mcl{J}:=\ker\rho$ has the form $(\mcl{J\cap A})\rtimes_{\alpha,r} G$, \cf{}Proposition \ref{appendix3}. This implies that $E_{\mcl{A/I}}(\rho(\mcl{J}))=E_{\mcl{A}}(\mcl{J})+\mcl{I}=0$.  Hence $E_{\mcl{A}}(\mcl{J})\subseteq \mcl{I}$. By Proposition \ref{appendix3} we also have that $\mcl{J\cap A}\subseteq E_{\mcl{A}}(\mcl{J})$. We conclude
$$\ker\rho=(\mcl{J\cap A})\rtimes_{\alpha,r} G\subseteq{\rm Ideal}[E_{\mcl{A}}(\mcl{J})]\subseteq \mcl{I}\rtimes_{\alpha,r} G.$$

{(iv) Residual intersection property:} Fix any $G$-invariant ideal $\mcl{I}$ of $\mcl{A}$ and any ideal $\mcl{J}$ of $\mcl{A/I}\rtimes_{\alpha,r} G$ with zero intersection with $\mcl{A/I}$. By Proposition \ref{appendix1} we have the commuting diagram
$$
\xymatrix{0 \ar[r] & \mcl{I}\rtimes_{\alpha,r}G \ar[r]^-{\iota} \ar[d]^-{E_\mcl{I}} & \mcl{A}\rtimes_{\alpha,r}G  \ar[r]^-{\rho} \ar[d]^-{E_\mcl{A}}& \mcl{A/I}\rtimes_{\alpha,r}G \ar[r] \ar[d]^-{E_{\mcl{A/I}}}& 0\\
0 \ar[r] & \mcl{I} \ar[r] & \mcl{A} \ar[r] & \mcl{A/I} \ar[r] & 0}
$$
By assumption $\mcl{J}_1:=\rho^{-1}(\mcl{J})$ has the form $(\mcl{J}_1\cap \mcl{A})\rtimes_{\alpha,r} G$, \cf{}Proposition \ref{appendix3}. Hence
\begin{align*}
E_{\mcl{A/I}}(\mcl{J})&=E_{\mcl{A/I}}(\rho(\mcl{J}_1))=E_{\mcl{A}}(\mcl{J}_1)+\mcl{I}=(\mcl{J}_1\cap \mcl{A})+\mcl{I}\\
&=E_{\mcl{A}}(\mcl{J}_1\cap \mcl{A})+\mcl{I}=E_{\mcl{A/I}}(\rho(\mcl{J}_1\cap \mcl{A}))\subseteq E_{\mcl{A/I}}(\rho(\mcl{J}_1)\cap \rho(\mcl{A}))\\
&=E_{\mcl{A/I}}(\mcl{J}\cap \mcl{A/I})=\mcl{J}\cap \mcl{A/I}=0
\end{align*}
By the faithfulness of the conditional expectation (on positive elements) we conclude that $\mcl{J}=0$.
\end{proof}

\begin{remark}
Exactness of a partial action is somehow mysterious because there are no concrete examples of a partial action that is not exact (even though existence has been established). Nevertheless, exactness plays an important role as we have seen above. Here is another result (\cite[Proposition 2.2.4]{Sie2}), proved by the second named author, relying heavily on exactness:

\emph{Let $(\mcl{A},G,\alpha)$ be a partial dynamical system. The action $\alpha$ is exact and $\mcl{A}\rtimes_{\alpha,r} G\cong \mcl{A}\rtimes_{\alpha} G$ if, and only if, $\mcl{A/I}\rtimes_{\alpha,r} G\cong \mcl{A/I}\rtimes_{\alpha} G$ for every $G$-invariant ideal $\mcl{I}$ in $\mcl{A}$.}
\end{remark}

If $\mcl{A}$ is a \cst-algebra, let ${\rm Prim}\,\mcl{A}$ denote its primitive ideal space and $\hat{\mathcal{A}}$ it spectrum. Moreover, if $\mcl{J}$ is a closed two-sided ideal of $\mcl{A}$, then ${\rm supp}\, \mcl{J}$ will denote the subset $\{\mcl{I}\in {\rm Prim}\,\mcl{A}\colon \mcl{J}\nsubseteq \mcl{I} \}$ and $\hat{\mathcal{A}}^\mcl{J} = \{ [\pi] \in \hat{\mathcal{A}} : \pi (\mcl{J}) \neq 0 \}$. Following Lebedev in \cite{Leb}, we define the map $\newfunk{\theta_t}{{\rm supp}\, \mcl{D}_{t^{-1}}}{{\rm supp}\, \mcl{D}_t}$ as follows; for any point $x \in {\rm supp}\, \mcl{D}_{t^{-1}}$ such that $x = {\rm ker}\, \pi$, where $[\pi] \in \hat{\mathcal{A}}^{\mcl{D}_{t^{-1}}}$, we set 
$$\theta_t(x) = \ker\, (\pi\circ\alpha_{t^{-1}}), \ \ \ t\in G.$$
Notice that we then have that $(\theta_t)_{t\in G}$ is an partial action of $G$ by partial homeomorphisms of ${\rm Prim}\,\mcl{A}$ {(see Appendix \ref{appendixA2} for more details).}

\begin{definition}
Let $(\mcl{A},G,\alpha)$ be a partial dynamical system.
\begin{itemize}
\item[(i)] The partial action $\alpha$ is \emph{topologically free} if for any finite set $F$ in $G\setminus\{e\}$ the union
$$\bigcup_{t\in F} \big\{x \in {\rm supp}\, \mcl{D}_{t}: \theta_{t} (x)=x   \big\}$$ 
has empty interior, \cf{}\cite{Leb}.
\item[(ii)] The partial action $\alpha$ is \emph{{residually topologically} free} if the induced action of $G$ on $\mcl{A/I}$ is topologically free for every $G$-invariant ideal $\mcl{I}$ of $\mcl{A}$.
\end{itemize}
\end{definition}

In \cite{Leb}, Lebedev shows that a {residually topologically} free action has always the residual intersection property (see Appendix \ref{appendixA2} for details). Hence, we have

\begin{corollary}
\label{cor1}
Let $(\mcl{A},G,\alpha)$ be a partial dynamical system. Suppose that the action is exact and {residually topologically} free. Then there is a one-to-one correspondence between ideals in $\mcl{A}\rtimes_{\alpha,r}G$ and $G$-invariant ideals in $\mcl{A}$.
\end{corollary}

There is a useful consequence of Corollary \ref{cor1} regarding $G$-prime ideals. Recall that a $G$-invariant ideal $\mcl{I}$ in a \cst-algebra $\mcl{A}$ is called {\em $G$-prime} (resp. {\em prime} if $G$ is the trivial group) if for any pair of $G$-invariant ideals $\mcl{J}$, $\mcl{K}$ of $\mcl{A}$ with $\mcl{J}\cap \mcl{K}\subseteq \mcl{I}$ we have $\mcl{J}\subseteq \mcl{I}$ or $\mcl{K}\subseteq \mcl{I}$, \cf\cite{EchLac}. Let $\mcl{I}(\mcl{A})$ denote the set of closed two-sided ideals in $\mcl{A}$. Imposing that the set $\{\mcl{J}\in  \mcl{I}(\mcl{A}) \colon \mcl{I}\not\subseteq \mcl{J}\}$ is open for any $\mcl{I}\in \mcl{I}(\mcl{A})$ defines a sub-basis for the Fell-topology on $\mcl{I}(\mcl{A})$. This topology induces topologies on the set of prime and $G$-prime ideals in $\mcl{I}(\mcl{A})$.
\begin{corollary}
\label{cor3}
Let $(\mcl{A},G,\alpha)$ be a partial dynamical system. The map 
$$\mcl{J}\mapsto \mcl{J}\cap \mcl{A}, \ \ \ \mcl{J}\in \mcl{I}(\mcl{A}\rtimes_{\alpha,r}G),$$
restricts to a continuous map from the space of prime ideals of $\mcl{A}\rtimes_{\alpha,r}G$ to the space of $G$-prime ideals in $\mcl{A}$. Moreover, if the action of $G$ on $\mcl{A}$ is exact and {residually topologically} free this restriction is a homeomorphism.
\end{corollary}
\begin{proof}

The second statement is contained in \cite{EchLac} for ordinary crossed products and it is evident that the proof generalizes. The same hold for the fact that the map $\mcl{J}\mapsto \mcl{J}\cap \mcl{A}$ is continuous. We conclude that the restriction is continuous (provided it is well defined).

The fact that it is well defined follows from: The ideal $\mcl{J}\cap \mcl{A}$ is $G$-prime for any prime ideal $\mcl{J}$ in $\mcl{A}\rtimes_{\alpha,r} G$. The proof is contained in \cite{EchLac} omitting one detail. If $\mcl{I}$ and $\mcl{K}$ are $G$-invariant ideals in $\mcl{A}$ then $(\mcl{I}\rtimes_{\alpha,r} G)\cap (\mcl{K}\rtimes_{\alpha,r}G) \subseteq (\mcl{I}\cap \mcl{K})\rtimes_{\alpha,r}G$. To see this we use that the intersection of two closed two-sided ideals equals their product and an approximation argument. It is then sufficient to show that if $a\delta_t\in \mcl{I}\rtimes_{\alpha,r} G$ and $b\delta_t\in \mcl{K}\rtimes_{\alpha,r} G$ then $(a\delta_t)(b\delta_s)\in  \mcl{(IK)}\rtimes_{\alpha,r} G$. But this is evident from  $(a\delta_t)(b\delta_s)=\alpha_t(\alpha_{t^{-1}}(a)b)\delta_{ts}$.
\end{proof}

\begin{remark} If a partial action of a discrete group $G$ on a \cst-algebra $\mcl{A}$ is exact then $\bigcap_i (\mcl{I}_i\rtimes_{\alpha,r} G)= (\bigcap_i\mcl{I}_i)\rtimes_{\alpha,r} G$ for any family $(\mcl{I}_i)$ of $G$-invariant ideals, \cf\cite{Sie2}.
\end{remark}

Let $x$ be a fixed element in ${\rm Prim}\,\mcl{A}$. Using the partial action $(\theta_t)_{t\in G}$ of $G$ on ${\rm Prim}\,\mcl{A}$ define $Gx:=\bigcup_{t\in G \textrm{ for which } x \in {\rm supp}\, \mcl{D}_{t^{-1}}}\{\theta_{t} (x)\}$.
The then {\em quasi-orbit space} $\Oo({\rm Prim}\,\mcl{A})$ of ${\rm Prim}\,\mcl{A}$ is defined as the quotient space ${\rm Prim}\,\mcl{A}/\sim$ by the equivalence relation $$x\sim y\Leftrightarrow \overline{G x}=\overline{G y}.$$
Recall that for $\mcl{A}=\czero(X)$ the partial action $\alpha$ is induced by a partial action $\theta$ of $G$ on $X$, i.e., a collection of open sets $(U_t)_{t\in G}$ and a collection $(\theta_t)_{t\in G}$ of homeomorphisms $\newfunk{\theta_t}{U_{t^{-1}}}{U_t}$ such that $U_e=X$ and $\theta_{st}$ extends  $\theta_{s}\circ\theta_{t}$, \cf{}\cite{McClan, Exe, QuiRae} and Appendix \ref{appendixA2}. Then $\alpha$ is given by $\alpha_t(f)(x) := f(\theta_{t^{-1}}(x))$, $f \in \czero(U_{t^{-1}})$. So, here the ideals are $\mcl{D}_t = \czero(U_t)$. As the canonical homeomorphism from  ${\rm Prim}\,\mcl{A}$ to $X$ is $G$-equivariant we identify the actions on ${\rm Prim}\,\mcl{A}$ and $X$. The next corollary and is proof is a generalisation of \cite[Lemma 2.5 and Corollary 2.6]{EchLac}.

\begin{corollary}
Let $(\mcl{A},G,\alpha)$ be a partial dynamical system with $\mcl{A}$ abelian and separable. Suppose that the action of $G$ on $\mcl{A}$ is exact and {residually topologically} free. Then the space of prime ideals of $\mcl{A}\rtimes_{\alpha,r}G$ is homeomorphic to the quasi-orbit space of ${\rm Prim}\,\mcl{A}$.
\end{corollary}

\begin{proof} We show that the following map is a homeomorphism:
$$[x]\mapsto \mcl{I}_x\rtimes_{\alpha,r} G, \ \ \ \mcl{I}_x:=\bigcap_{\mcl{I}\in G x}\mcl{I}, \ \ \  [x]\in \Oo({\rm Prim}\,\mcl{A}).$$

Well defined: Fix any $[x]\in \bigcap_{\mcl{I}\in G x}\mcl{I}$. It is evident that $\mcl{I}_x$ is a closed two-sided ideal in $\mcl{A}$. Let $\pi: \mcl{A}\to B(H)$ be the irreducible representation corresponding to $x$, i.e.\ $x=\ker\pi$. Since $\mcl{A}=\czero(X)$ for some locally compact Hausdorff space \cite[II.6.2.9]{Bla} ensures that $\pi(f)=f(y)$ for some $y\in X$, so $x=\czero(X\setminus\{y\})$. Hence 
$$\theta_t(x)=\ker(\pi\circ\alpha_{t^{-1}})=\{f:\alpha_{t^{-1}}(f)(y)=0\}=\czero(X\setminus\{\theta_t(y)\}).$$
By continuity we obtain that $\mcl{I}_x=\czero(X\setminus\overline{ Gy})$. This ideal is $G$-invariant since $\overline{ Gy}$ is $G$-invariant (and the complement of a $G$-invariant set is also $G$-invariant). Let us verify $\overline{Gy}$ is $G$-invariant for completeness. Fix any $s\in G$ and any $\theta_t(y) \in Gy\cap U_{s^{-1}}$, where $\mcl{D}_s = \czero(U_s)$. Since $\theta_t(y) \in U_t\cap U_{s^{-1}}$ \cite[Lemma 1.2]{QuiRae} ensures that $y\in \theta_{t^{-1}}(U_t\cap U_{s^{-1}})=U_{t^{-1}}\cap U_{{(st)}^{-1}}$. In particular $\theta_s(\theta_t(y))=\theta_{st}(y)$. We conclude $\theta_s(Gy\cap U_{s^{-1}})\subseteq Gy$ making $Gy$ invariant. The fact that $\overline{Gy}$ is invariant is a simple limit argument used on $\overline{Gy}\cap U_{s^{-1}} \subseteq \overline{Gy\cap U_{s^{-1}}}$. Let us verify that $\mcl{I}_x$ is $G$-prime. Suppose that $U,V$ are closed $G$-invariant subsets of $X$. If $\czero(X\setminus U)\cap \czero(X\setminus V)\subseteq \mcl{I}_x$ it follows that $y\in U\cup V$, hence either $\czero(X\setminus U)\subseteq \mcl{I}_x$ or $\czero(X\setminus V)\subseteq \mcl{I}_x$. It now follows from Corollary \ref{cor3} that $\mcl{I}_x\rtimes_{\alpha,r} G$ is prime.

Surjectivity (\cite[Lemma 2.5]{EchLac}): Let $\mcl{J}$ be any prime ideal in $\mcl{A}\rtimes_{\alpha,r} G$. By Corollary \ref{cor3} the ideal $\mcl{I}=\mcl{J}\cap \mcl{A}$ is $G$-prime. As $\mcl{I}$ is $G$-invariant $\mcl{I}=\czero(X\setminus V)$ for some nonempty closed $G$-invariant set $V$ in $X$. Let $\newfunk{\rho}{{\rm Prim}\,\mcl{A}}{T}$ denote the quotient map into the quasi-orbit space $T:=\Oo({\rm Prim}\,\mcl{A})$. To show that $V=\overline{Gx}$ for some $x\in X$ we verify that $F:=\rho(V)$ is the closure of a single point in $T$. Notice that $T$ is totally Baire (every intersection of an open and a closed subset is a Baire space) and second countable. For such a space a non-empty closed subset $F$ is the closure of a simple point if and only if $F$ is not a union of two proper closed subsets, \cf{}\cite[Lemma p.\ 222]{Gre}. However such two proper closed subsets give raise to two closed $G$-invariant subsets $U_1, U_2\subsetneq V$ with $U_1\cup U_2=V$ and hence two $G$-invariant ideals $\mcl{I}_1, \mcl{I}_1 \supsetneq \mcl{I}$ with $\mcl{I}_1\cap \mcl{I}_2=\mcl{I}$. This contradicts $G$-primeness of $\mcl{I}$.

Injectivity: Fix any two prime ideals  $\mcl{I}_x\rtimes_{r,\alpha} G$ and  $\mcl{I}_y\rtimes_{r,\alpha} G$ in  $\mcl{A}\rtimes_{r,\alpha} G$ such that $\mcl{I}_x\rtimes_{r,\alpha} G = \mcl{I}_y\rtimes_{r,\alpha} G$. By Proposition \ref{appendix3} we obtain $\czero(X\setminus\overline{ Gx})=\mcl{I}_x=\mcl{I}_y=\czero(X\setminus\overline{ Gy})$. We conclude that $x\sim y$.

\end{proof}

\begin{corollary}[Lebedev \cite{Leb}]
Let $(\mcl{A},G,\alpha)$ be a partial dynamical system. Suppose that the action is minimal (i.e{,} $\mcl{A}$ does not contain any non-trivial $G$-invariant ideals) and topologically free. Then the crossed product $\mcl{A}\rtimes_{\alpha,r} G$ is simple.
\end{corollary}

The notation of topological freeness for partial actions is well known. We recall it in Appendix \ref{appendixA2}  and show that the following equivalent conditions hold:
\begin{proposition}
\label{abeliantopfree}
Let $(\mcl{A},G,\alpha)$ be a partial dynamical system with $\mcl{A}$ abelian. Then the following properties are equivalent:
\begin{enumerate}[(i)]
\item $\alpha$ is topologically free
\item $\|\alpha_t|_{\mcl{K}}-id|_\mcl{K}\|=2$ for every $\alpha_t$-invariant ideal $\mcl{K}$ in $D_{t^{-1}}$ (and $t\neq e$)
\item $\inf\{\|x\alpha_{t}(x)\|: x\in \mcl{K}_+, \|x\|=1\}=0$ for every nonzero ideal $\mcl{K}$ in $\mcl{D}_{t^{-1}}$ (and $t\neq e$)
\item $\inf\{\|{x}( {a}\delta_t){x}\|: x\in \mcl{B}_+, \|x\|=1\}=0$ for every ${a}\in{\mcl{D}_t}$, and every nonzero hereditary \cst-algebra $\mcl{B}$ in $\mcl{A}$ (and $t\neq e$)\\
\end{enumerate}
\end{proposition}
In \cite{Ell}, Elliott defines an automorphism $\alpha$ of a \cst-algebra $\mcl{A}$ to be \emph{properly outer} if for any nonzero $\alpha$-invariant ideal $\mcl{I}$ of $\mcl{A}$ and any inner automorphism $\beta$ of $\mcl{I}$, $$\|\alpha|_{\mcl{I}}-\beta\|=2.$$
Proper outerness for dynamical systems has been vastly studied in \cite{Kis, OlePed3, ArcSpi, Sie2, Sie} among others. We will address in one of our upcoming works how one might generalise the notion proper outerness beyond ordinary crossed products.

In the rest of the section, we generalize \cite{ExeLacQui}, Theorem 2.6, to the non-abelian case using condition (iv) of Proposition \ref{abeliantopfree}. We first need the following lemma.

\begin{lemma}
\label{exellemma}
Let $(\mcl{A},G,\alpha)$ be a partial dynamical system. Suppose that for every $t\neq e$, every ${a}\in{\mcl{D}_t}$, and every nonzero hereditary \cst-algebra $\mcl{B}$ in $\mcl{A}$
$$\inf\{\|{x}( {a}\delta_t){x}\|: x\in \mcl{B}_+, \|x\|=1\}=0.$$
Then for every $b\in (\mcl{A}\rtimes_{\alpha,r} G)_+$ and every $\varepsilon>0$ there exist a positive contraction $x\in \mcl{A}$ satisfying
$$\|xE_{\mcl{A}}(b)x-xbx\|< \varepsilon, \ \ \ \|xE_{\mcl{A}}(b)x\|> \|E_{\mcl{A}}(b)\|-\varepsilon.$$
\end{lemma}

 \begin{proof}
Fix $b\in (\mcl{A}\rtimes_{\alpha,r} G)_+$ and $\varepsilon>0$. 

{(i) We may assume $b\in \mcl{L}$:} Since $\mcl{L}$ is dense in $\mcl{A}\rtimes_{\alpha,r} G$ there exist ${c}\in (\mcl{A}\rtimes_{\alpha,r} G)_+\cap \mcl{L}$ such that $\|{c}-b\|<\varepsilon$. Find a positive contraction $x\in \mcl{A}$ satisfying
$\|xE_{\mcl{A}}({c})x-x{c}x\|< \varepsilon$, and $\|xE_{\mcl{A}}({c})x\|> \|E_{\mcl{A}}({c})\|-\varepsilon$. Then
\begin{align*}
\|xE_{\mcl{A}}(b)x-xbx\|&\leq \|xE_{\mcl{A}}(b)x-xE_{\mcl{A}}({c})x\|+\|xE_{\mcl{A}}({c})x-x{c}x\|\\
&\ \ \ +\|x{c}x-xbx\| < 3\varepsilon,\\
\|E_{\mcl{A}}(b)\|
&< \|E_{\mcl{A}}(b-{c})\|+\varepsilon+\|xE_{\mcl{A}}({c}-b)x\|+\|xE_{\mcl{A}}(b)x\|\\
&< \|xE_{\mcl{A}}(b)x\|+3\varepsilon.
\end{align*}

{(ii) We may assume $E_{\mcl{A}}(b)$ has norm one:} If $b=0$ then any positive contraction $x\in\mcl{A}$ works. For $b\neq 0$ define ${c}:=\frac{b}{\|E_{\mcl{A}}(b)\|}$ using $E_{\mcl{A}}$ is faithful. Find a positive contraction $x\in \mcl{A}$ satisfying
$\|xE_{\mcl{A}}({c})x-x{c}x\|< \frac{\varepsilon}{\|E_{\mcl{A}}(b)\|}$, and $\|xE_{\mcl{A}}({c})x\|> \|E_{\mcl{A}}({c})\|-\frac{\varepsilon}{\|E_{\mcl{A}}(b)\|}$. Then
$$\|xE_{\mcl{A}}(b)x-xbx\|< \varepsilon, \ \ \ \|xE_{\mcl{A}}(b)x\|> \|E_{\mcl{A}}(b)\|-\varepsilon.$$

{(iii) It is enough to show that:} For every $\varepsilon>0$, every $b_0\in \mcl{A}_+$, with $\|b_0\|=1$, every finite set $F\subseteq G\setminus\{e\}$, and every sequence of elements $(b_t)_{t\in F}$, with $b_t\in \mcl{D}_{t}$, there exist $x\in \mcl{A}_+$ such that 
$$\|x\|=1, \ \ \ \|xb_0x\|>\|b_0\|-\varepsilon, \ \ \ \|{x}( b_t\delta_t){x}\|<\varepsilon, \ \ \ t\in F.$$
By $(i)-(ii)$ we may assume that $b={b_0}+\sum_{t \in F} b_t\delta_t$ for some finite set $F\subseteq G\setminus\{e\}$, where $b_0=E_{\mcl{A}}(b)$ is positive (since $E_{\mcl{A}}$ is positive) and has norm one.
Using the assumption on $b_0$ and the sequence  $(b_t)_{t\in F}$ choose $x\in \mcl{A}_+$ such that 
$$\|x\|=1, \ \ \ \|xb_0x\|>\|b_0\|-\varepsilon, \ \ \ \|{x}( b_t\delta_t){x}\|<\varepsilon, \ \ \ t\in F.$$
We conclude that
\begin{align*}
\|xE_{\mcl{A}}(b)x\|&>\|E_{\mcl{A}}(b)\|-\varepsilon,\\
\|xE_{\mcl{A}}(b)x-xbx\|&=\|x(b-b_0)x\|\leq \sum_{t\in F}\|{x}( b_t\delta_t){x}\|< |F| \varepsilon.
\end{align*}

{(iv) Finishing the proof:} Fix $\varepsilon>0$, $b_0\in \mcl{A}_+$, with $\|b_0\|=1$, a finite set $F\subseteq G\setminus\{e\}$, and a sequence of elements $(b_t)_{t\in F}$, with $b_t\in \mcl{D}_{t}$. Let $\newfunk{f}{[0,1]}{[0,1]}$ be a continuous increasing function taking the value zero on $[0,1-\varepsilon]$ and one on $[1-\frac{\varepsilon}{2},1]$. Define $x_0:=f(b_0)$ and set
$$\mcl{B}_1:=\{x \in \mcl{A}\colon xx_0=x_0x=x\}.$$
If $\newfunk{g}{[0,1]}{[0,1]}$ is a continuous increasing function equal to zero on $[0,1-\frac{\varepsilon}{2}]$ and one on $[1-\frac{\varepsilon}{4},1]$ then $gf=fg=g$ and $\|g(b_0)\|=1$. We conclude that $\mcl{B}_1$ is nonzero. Using the \cst-norm identify it follows that $L=\{x\in \mcl{A}\colon x^*x\in \mcl{B}_1\}$ is a closed left ideal in $\mcl{A}$. Verifying that $\mcl{B}_1=L\cap L^*$ we obtain that $\mcl{B}_1$ is hereditary. Write $F:=\{t_1,\dots, t_n\}$. Since $b_{t_1}\in{\mcl{D}_{t_1}}$ and $\mcl{B}_1$ is  nonzero hereditary in $\mcl{A}$
$$\inf\{\|{x}( b_{t_1}\delta_{t_1}){x}\|: x\in (\mcl{B}_1)_+, \|x\|=1\}=0.$$
Select $x_1'\in (\mcl{B}_1)_+$ such that $\|x_1'\|=1$ and $\|{x_1'}(b_{t_1}\delta_{t_1}){x_1'}\|<\varepsilon$. Let $\newfunk{h}{[0,1]}{[0,1]}$ be a continuous increasing function equal to the identity on $[0,1-\varepsilon]$ and one on $[1-\frac{\varepsilon}{2},1]$. With $x_1:=h(x_1')$ we have (since $gh=hg=g$) that 
\begin{align*}
&x_1\in (\mcl{B}_1)_+, \ \ \ \|x_1\|=1, \ \ \ \|{x_1}(b_{t_1}\delta_{t_1}){x_1}\|<(2+\|b_{t_1}\|)\varepsilon,\\
&\mcl{B}_2:=\{x \in{\mcl{B}_1}\colon xx_1=x_1x=x\}\ni g(x_1')\neq 0.
\end{align*}
Since $b_{t_2}\in{\mcl{D}_{t_2}}$ and $\mcl{B}_2$ is  nonzero hereditary in $\mcl{A}$
$$\inf\{\|{x}( b_{t_2}\delta_{t_2}){x}\|: x\in (\mcl{B}_2)_+, \|x\|=1\}=0.$$
Repeating the procedure above we obtain a sequence of hereditary \cst-algebras $$\mcl{B}_1\supseteq \mcl{B}_2\supseteq \cdots \supseteq \mcl{B}_n,$$ and positive elements of norm one in them, $x_1$, $x_2$, $\dots$, $x_n$, fulfilling that
$$x_i\in (\mcl{B}_i)_+, \ \ \ \|x_i\|=1, \ \ \ \|{x_i}(b_{t_i}\delta_{t_i}){x_i}\|<(2+\|b_{t_i}\|)\varepsilon.$$
With $x:=x_n$ we have that $xx_i=x_ix=x$ for $0\leq i<n$. It follows that
$$x\in \mcl{A}_+, \ \ \ \|x\|=1, \ \ \ \|{x}(b_{t_i}\delta_{t_i}){x}\|<(2+\|b_{t_i}\|)\varepsilon, \ \ \ i=1,\dots, n.$$
Finally, since $f(t)^2t\geq (1-\varepsilon)f(t)^2$ for every $t\in [0,1]$, we get
$$xx_0b_0x_0x\geq (1-\varepsilon)xx_0x_0x, \ \ \ \|xb_0x\|>\|b_0\|-\varepsilon.$$
 \end{proof}

\begin{theorem}
\label{mainthorem2}
Let $(\mcl{A},G,\alpha)$ be a partial dynamical system. Suppose that for every $t\neq e$, every ${a}\in{\mcl{D}_t}$, and every nonzero hereditary \cst-algebra $\mcl{B}$ in $\mcl{A}$
$$\inf\{\|{x}( {a}\delta_t){x}\|: x\in \mcl{B}_+, \|x\|=1\}=0.$$
Then every nonzero ideal in $\mcl{A}\rtimes_{\alpha, r} G$ has a nonzero intersection with $\mcl{A}$.
\end{theorem}
\begin{proof}
The proof by Exel, Laca and Quigg in \cite{ExeLacQui} for partial crossed product with $\mcl{A}$ abelian generalises to the non-abelian case by means of Lemma \ref{exellemma}.
\end{proof}

\begin{corollary}
\label{cor2.10}
Let $(\mcl{A},G,\alpha)$ be a partial dynamical system. Suppose that the action is exact and that for every $t\neq e$, every $G$-invariant ideal $\mcl{I}$ in $\mcl{A}$, every ${a}\in{\mcl{D}_t}/(\mcl{D}_t\cap \mcl{I})$, and every nonzero hereditary \cst-algebra $\mcl{B}$ in $\mcl{A}/\mcl{I}$
$$\inf\{\|{x}( {a}\delta_t){x}\|: x\in \mcl{B}_+, \|x\|=1\}=0.$$
Then there is a one-to-one correspondence between ideals in $\mcl{A}\rtimes_{\alpha,r}G$ and $G$-invariant ideals in $\mcl{A}$.
\end{corollary}

\begin{proof}
Using Theorem \ref{mainthorem2} we obtain that $\alpha$ has the residual intersection property. The desired correspondence follows from Theorem \ref{mainthorem}.
\end{proof}

\section{Pure infiniteness of partial crossed products}\label{preliminaries2}
\label{section3}
We recall the definition of purely infinite \cst-algebras with the ideal property and the notion of paradoxical actions.

Let $\mcl{A}$ be a \cst-algebra, and let $a,b$ be positive elements in $\mcl{A}$. We say that $a$ is \emph{Cuntz below} $b$, denoted $a\precsim b$, if there exists a sequence $(r_n)$ in $\mcl{A}$ such that $r_n^*br_n\to a$.
More generally for $a\in M_n(\mcl{A})_+$ and $b\in M_m(\mcl{A})_+$ we write $a\precsim b$ if there exists a sequence $(r_n)$ in $M_{m,n}(\mcl{A})$ with $r_n^*br_n\to a$. For $a\in M_n(\mcl{A})$ and $b\in M_m(\mcl{A})$ let $a\oplus b$ denote the element $\textrm{diag}(a,b)\in M_{n+m}(\mcl{A})$. A nonzero positive element $a$ in $\mcl{A}$ is \emph{properly infinite} if $a\oplus a\precsim a$.

A \cst-algebra $\mcl{A}$ is \emph{purely infinite} if there are no characters on $\mcl{A}$ and if for every pair of positive elements $a,b$ in $\mcl{A}$ such that $b$ belongs to the ideal in $\mcl{A}$ generated by $a$, one has $b\precsim a$. Equivalently, a \cst-algebra $A$ is purely infinite if every non-zero positive element $a$ in $A$ is properly infinite, \cf\cite[Theorem 4.16]{KirRor}.

\begin{definition}
A \cst-algebra \mcl{A} has \emph{the ideal property} if projections in $\mcl{A}$ separate ideals in $\mcl{A}$, \ie, whenever $\mcl{I,J}$ are ideals in $\mcl{A}$ such that $\mcl{I}\nsubseteq \mcl{J}$, then there is a projection in $\mcl{I}\setminus (\mcl{I\cap  J})$.
\end{definition}

The next result generalises the work in \cite{RorSie} on ordinary crossed products. The proof is included for completeness.

\begin{theorem}
\label{the4.1}	
Let $(\mcl{A},G,\alpha)$ be a partial dynamical system. Suppose that the action is exact and that for every $t\neq e$, every $G$-invariant ideal $\mcl{I}$ in $\mcl{A}$, every ${a}\in{\mcl{D}_t}/(\mcl{D}_t\cap \mcl{I})$, and every nonzero hereditary \cst-algebra $\mcl{B}$ in $\mcl{A}/\mcl{I}$
$$\inf\{\|{x}( {a}\delta_t){x}\|: x\in \mcl{B}_+, \|x\|=1\}=0.$$
Suppose also that $\mcl{A}$ has the ideal property. Then the following statements are equivalent
\begin{enumerate}[(i)]
\item Every nonzero positive element in $\mcl{A}$ is properly infinite in
$\mcl{A}\rtimes_{\alpha,r} G$. 
\item The \cst-algebra $\mcl{A}\rtimes_{\alpha,r} G$ is purely infinite.
\item Every nonzero hereditary sub-\cst-algebra in any quotient of $\mcl{A}\rtimes_{\alpha,r} G$ contains an infinite projection. 
\end{enumerate}
\end{theorem}
\begin{proof}
The implications $(iii)\Rightarrow (ii) \Rightarrow (i)$ are valid for any partial dynamical system, \cf{}\cite{KirRor}. For $(i)\Rightarrow (iii)$ let $\mcl{J}$ be any ideal in ${\mcl{A}\rtimes_{\alpha,r}G}$ and let $\mcl{B}$ be any nonzero hereditary sub-\cst-algebra in the quotient $(\mcl{A}\rtimes_{\alpha,r} G)/\mcl{J}$. We show that $\mcl{B}$ contains an infinite projection. By assumption on $\alpha$ and Corollary \ref{cor2.10} we have that $(\mcl{A}\rtimes_{\alpha,r} G)/\mcl{J}\cong (\mcl{A/I})\rtimes_{\alpha,r} G$ for $\mcl{I=J\cap A}$. Select a nonzero positive element $b$ in $\mcl{B}$ such that $\|E_{\mcl{A/I}}(b)\|=1$. By Lemma \ref{exellemma} there exist a positive contraction $x\in \mcl{A/I}$ satisfying 
$$\|xE_{\mcl{A/I}}(b)x-xbx\|< 1/4, \ \ \ \|xE_{\mcl{A/I}}(b)x\|> \|E_{\mcl{A/I}}(b)\|-1/4=3/4.$$
With $a = (xE_{\mcl{A/I}}(b)x-1/2)_+$ we claim that $0\neq a\precsim xbx \precsim b$. Indeed, the element $a$ is nonzero because $\|xE_{\mcl{A/I}}(b)x\|>1/2$, and $a\precsim xbx$ holds since $\|xE_{\mcl{A/I}}(b)x-xbx\| < 1/2$, \cf{}\cite[Proposition 2.2]{Ror}. By the assumption that $\mcl{A}$ has the ideal property we can find a projection $q\in \mcl{A}$ that belongs to the ideal in $\mcl{A}$ generated by the preimage of $a$ in $\mcl{A}$ but not to $\mcl{I}$. Then $q+\mcl{I}$ belongs to the ideal in $\mcl{A/I}$ generated by $a$, whence $q+\mcl{I} \precsim a \precsim b$ in $\mcl{A/I} \rtimes_r G$ (because $a$ is properly infinite by assumption), \cf{}\cite[Proposition 3.5 (ii)]{KirRor}. From the comment after \cite[Proposition 2.6]{KirRor} we can find $z\in\mcl{A/I}\rtimes_r G$ such that $q+\mcl{I}=z^*bz$. With $v=b^{1/2}z$ it follows that $v^*v= q+\mcl{I}$, whence $p := vv^* =
b^{1/2}zz^*b^{1/2}$ is a projection in $\mcl{B}$, which is equivalent to $q+\mcl{I}$. By the assumption $q$ is properly infinite, and hence so is $q+\mcl{I}$ (since the relation $a\otimes a\precsim a$ passes to quotients) and $p$.
\end{proof}

We now introduce a geometrical condition sufficient for the pure infiniteness of a large class of partial crossed product \cst-algebras. It is not clear if this geometrical condition is also a necessary one.

\begin{definition}
\label{paradox}
Let $(\czero(X),G,\alpha)$ be a partial dynamical system, together with the corresponding collection of homeomorphisms $(\newfunk{\theta_t}{X_{t^{-1}}}{X_t})_{t \in G}$, inducing $\alpha$, \cf{}Appendix \ref{appendixA2}. Let $\mathbb{E}$ denote a family of subsets of $X$. A non-empty set $V\subseteq X$ is called \emph{$(G,\mathbb{E})$-paradoxical} if there exist sets $V_1,V_2\dots,V_{n+m}\in \mathbb{E}$ and elements $t_1,t_2,\dots,t_{n+m}\in G$ such that
$$\bigcup_{i=1}^n V_i=\bigcup_{i=n+1}^{n+m} V_i=V, \ \ \ V_i\subseteq {X_{t_{i}^{-1}}}, \ \ \ \theta_{t_i}(V_i)\subseteq V, \ \ \ \theta_{t_k}(V_k)\cap\theta_{t_l}(V_l)=\emptyset, k\neq l.$$
We let $\tau_X$ denote the topology of $X$.
\end{definition}

\begin{theorem}
\label{the4.2}
Let $(\czero(X),G,\alpha)$ be a partial dynamical system. Suppose that the action is exact and {residually topologically} free and that $X$ is totally disconnected. Suppose also that every compact and open subset of $X$ is $(G,\tau_X)$-paradoxical. Then $\czero(X)\rtimes_{\alpha,r} G$ is purely infinite.
\end{theorem}
\begin{proof}
The proof consist of two parts. First we show that for any $(G,\tau_X)$-paradoxical, compact and open subset $U \subseteq X$ the projection ${1_U}$ is properly infinite in ${C_0(X)\rtimes_{\alpha,r}G}$. We then prove that proper infiniteness of such projections is enough to ensure pure infiniteness of $\czero(X)\rtimes_{\alpha,r} G$. We do not know {if} the second part follows from Theorem \ref{the4.1}.

{(i) Proper infiniteness:} Recall that $\alpha$ is induced by a collection of open sets $(X_t)_{t\in G}$ and a collection $(\theta_t)_{t\in G}$ of homeomorphisms $\newfunk{\theta_t}{X_{t^{-1}}}{X_t}$ such that $X_e=X$ and $\theta_{st}$ extends  $\theta_{s}\circ\theta_{t}$, \cf{}\cite{McClan, Exe, QuiRae} and Appendix \ref{appendixA2}. The partial action $\alpha$ of $G$ on $\czero(X)$ corresponding to $\theta$ is given by $\alpha_t(f)(x) := f(\theta_{t^{-1}}(x))$, $f \in \czero(X_{t^{-1}})$. So, here the ideals are $\mcl{D}_t = \czero(X_t)$. Let $(V_i,t_i)_{i=1}^{n+m}$ denote the system of open sets of $U$ and elements in $G$ witnessing the paradoxality of $U$. Find partitions of unity $(h_i)_{i=1}^n$ and $(h_i)_{i=n+1}^{n+m}$ for $U$ relative to the open covers $(V_i)_{i=1}^{n}$ and $(V_i)_{i=n+1}^{n+m}$, respectively. For each $i=1,\dots, n+m$ we have that the (compact) support of $h_i$ lies in $V_i\subseteq X_{t_i^{-1}}$. Hence $h_i\in \mcl{D}_{t_i^{-1}}$ and we can define $a_i:=\alpha_{t_i}(h_i^{1/2})$, for $i=1,\dots,n+m$, and the elements
$$x=\sum_{i=1}^na_i\delta_{t_i}, \ \ \ y=\sum_{i=n+1}^{n+m}a_i\delta_{t_i}, \ \ \ p={1_U}.$$
Using that $(a_t\delta_t)(a_s\delta_s)=\alpha_t(\alpha_{t^{-1}}(a_t)a_s)\delta_{ts}$ and $(a_t\delta_t)^*=\alpha_{t^{-1}}(a_t^*)\delta_{t^{-1}}$ it follows that
\begin{align*}
	(a_i\delta_{t_i})^*(a_j\delta_{t_j}) &= (\alpha_{t_i^{-1}}(a_i^*)\delta_{t_i^{-1}})(a_j\delta_{t_j})\\
	& = \alpha_{t_i^{-1}}(\alpha_{t_i}(\alpha_{t_i^{-1}}(a_i^*))a_j)\delta_{t_i^{-1}g_j}\\
	& = \alpha_{t_i^{-1}}(a_i^*a_j)\delta_{t_i^{-1}t_j}\\
	&=\left\{
	\begin{array}{lr}
	{h_i} &\mbox{ if } i=j\\
	0 &\mbox{ if } i\neq j.\\
	\end{array}\right.
\end{align*}
We obtain that $x^*x=y^*y=p$ and $y^*x=0$. Moreover, since
\begin{align*}
	{1_U}(a_i\delta_{t_i}) &= \alpha_{e}(\alpha_e(1_U)a_i)\delta_{et_i}\\
	&=(1_Ua_i)\delta_{t_i}\\
	&=a_i\delta_{t_i},
\end{align*}
we have that $px=x$, $py=y$. This implies that $xx^*+yy^*\leq p$, hence $p$ is properly infinite, \cf{}\cite{KirRor}.

{(ii) Pure infiniteness:} This part follows by inspection of the proof of Theorem \ref{the4.1}. The difference is that we do not assume that every nonzero element $a\in \mcl{A/I}$ is properly infinite (in $\mcl{A/I}\rtimes_{\alpha,r} G$) and we can therefore not conclude that the selected nonzero projection $q + \mcl{I}$ in the ideal in $\mcl{A/I}$ generated by $a$ fulfils that $q+\mcl{I} \precsim a$ (in $\mcl{A/I} \rtimes_r G$) using \cite[Proposition 3.5 (ii)]{KirRor}. Instead we find $q\in\mcl{A}$ fulfilling $0\neq q + \mcl{I}\precsim a$ as follows: Since $X$ {is} totally disconnected we get that $\mcl{A}$ and $\mcl{A/I}$ have real rank zero and every projection in $\mcl{A/I}$ lifts to a projection in $\mcl{A}$, \cf{}\cite[Proposition 1.1, Theorem 3.14]{BroPed}. We can therefore select any nonzero projections in the hereditary sub-\cst-algebra of $\mcl{A/I}$ generated by $a$ and lift it to a projections $q\in \mcl{A}$. Since $0\neq q + \mcl{I}$ is contained in the hereditary sub-\cst-algebra of $\mcl{A/I}$, $q + \mcl{I}\precsim a$, \cite[Proposition 2.7]{KirRor}.
\end{proof}

\section{Examples}
\label{examples}
\label{section4}
\begin{example}[The Cuntz algebra $\mathcal{O}_n$]
The \emph{Cuntz algebra}, denoted by $\mathcal{O}_n$, $n\geq 2$ is the universal \cst-algebra generated by isometries $s_1,s_2,\dots,s_n$ subject to the relation
$$
\sum_{i=1}^n s_is_i^* =1
$$

Let $\QQ_n$ be the group $\{\frac{p}{n^k}\colon p\in \ZZ, k\in \NN{\cup \{0\}}\}$ of $n$-adic rationals. We will denote by $\qqxz$ the semidirect product: $\qqxz = \{(r, k): r \in \QQ_n, k \in \ZZ\}$. The two group operations are given by: $(s, j)(r, k) = (\frac{r}{n^j}+s, j+k)$ and $(r, k)^{-1} = (-n^kr, -k)$. Hopenwasser constructs in \cite{Hop} a partial action $\alpha$ of $\QQ_n$ on the Cantor set $X$ such that $\ccont(X)\rtimes_{\alpha,r} G \cong \mathcal{O}_n$.
\begin{proposition}
Let $X$ be the Cantor set and let $(\ccont(X),\qqxz,\alpha)$ be the partial dynamical system described in \cite{Hop}. Then every clopen subset of $X$ is $(\qqxz,\tau_X)$-paradoxical.
\end{proposition}

\begin{proof}
Let $X$  denote the Cantor set based on $[0,1]$ where each $n$-adic rational $r$ (exempt 0 and 1) is replaced by a pair $r^-$, $r^+$, such that $r^-$ is the \emph{immediate predecessor} of $r^+$ (\ie{}there are no non-trivial elements $x$ in $X$ such that $r^-\leq x \leq r^+$). To simplify notation we will identify $0^+$ with $0$ and $1^-$ with $1$.

Let $U$ be any nonempty clopen subset of $X$. Recall that sets of the form $[(\frac{p}{n^k})^{+}, (\frac{q}{n^l})^{-}]$ (with $p,q\in \NN{\cup \{0\}}$, and $k,l\in\NN$) form a basis for the topology on $X$ consisting of compact and open sets. It follows that $U$ is a finite disjoint union of such sets (because we can cover $U$ by finitely many such subsets of $U$ and then remove the parts where they intersect). Since a disjoint union of compact open $(\qqxz,\tau_X)$-paradoxical sets is again $(\qqxz,\tau_X)$-paradoxical we can assume that $U=[(\frac{p}{n^k})^{+}, (\frac{q}{n^l})^{-}]$ for some $p,q\in \NN{\cup \{0\}}$, and $k,l\in\NN$. Rewriting the fractions allows us to assume $l=k$. Since $U$ is nonempty $p<q$. Suppressing the index $\pm$ we have that $U$ is a disjoint union of sets $[\frac{p}{n^k}, \frac{p+1}{n^k}]$, $\dots$, $[\frac{q-1}{n^k}, \frac{q}{n^k}]$. We can therefore assume that $q=p+1$. If $U\subseteq X_{s^{-1}}$ for some $s\in \qqxz$ it follows that $U$ is $(\qqxz,\tau_X)$-paradoxical if, and only if, $\theta_s(U)\subseteq X_s$ is $(\qqxz,\tau_X)$-paradoxical. Using the relations
\begin{align*}
	 X_{{(-\frac{1}{n^k},0)}^{-1}}&=[\frac{1}{n^k},1], & \theta_{(-\frac{1}{n^k},0)}\big(\big[\frac{p}{n^k}, \frac{p+1}{n^k}\big]\big)&=\big[\frac{p-1}{n^k}, \frac{p}{n^k}\big] \ \ ({\rm if }\ k\neq 0),\\
	 X_{(0,-k)^{-1}}&=\big[0, \frac{1}{n^k}\big] & \theta_{(0,-k)}\big(\big[0, \frac{1}{n^k}\big]\big)&=X.
\end{align*}
we can assume that $U=X$. It is evident that $X$ is $(\qqxz,\tau_X)$-paradoxical. In fact just like the crossed product associated to the action of the Baumslang-Solitar group on $\RR$, \cf{}\cite{KirSie}, we also have here an action that mimics both translation and scaling. By first shrinking two copies of $X$ and then translating one of subsets away from the other, we naturally obtain the paradoxical property of $X$.
\end{proof}

\end{example}

\begin{example}[The Cuntz-Krieger algebra ${\mathcal O}_A$]
Let $A=[a_{ij}]$ be  a $\{0,1\}$-valued $(n\times n)$-matrix with no zero rows. We define the algebra $\mathcal O_A$ to be the universal \cst-algebra generated by $n$ partial isometries $\{s_i\}_{i=1}^n$ satisfying that 
$$
\sum_j s_j s_j^* = 1, \quad \text{ and } \quad
\sum_j a_{ij} s_j s_j^* = s_i^* s_i \quad\text{ for }i = 1, 2, \ldots ,n.
$$
We have chosen to define $\mathcal O_A$ as a universal object, as in \cite{HueRae}, because it allows us to identify it with a partial crossed product. If the matrix $A$ satisfies Condition (I) of Cuntz and Krieger, which implies the uniqueness of the $C^*$-algebra $C^*(\{s_i\})$ provided that $s_i \neq 0$ for every $i$, we obtain the well known Cuntz-Krieger algebra introduced in \cite{CunKri}.
\begin{proposition}
Let  $X$ be a compact totally disconnected space and $G= \FF_n$ be the free group on $n$ generators $\{ g_1, g_2, \dots, g_n\}$.
Let $(\ccont(X),G,\alpha)$ denote the partial dynamical system described by \cite{ExeLacQui} such that $$\ccont(X)\rtimes_{\alpha,r} G \cong \mathcal O_A.$$
If $A$ is such that the action is exact and {residually topologically} free, then the following statements are equivalent
\begin{enumerate}[(i)]
\item  Every compact and open subset of $X$ is $(G,\tau_X)$-paradoxical
\item The \cst-algebra $\ccont(X)\rtimes_{\alpha,r} G$ is purely infinite.
\end{enumerate}
\end{proposition}
\begin{proof}
The implication $(i)\Rightarrow (ii)$ follows form Theorem \ref{the4.2}. For the converse implication $(ii)\Rightarrow (i)$ we need to recall the construction of the partial crossed product.  An \emph{infinite admissible path} is an infinite sequence $\mu = \mu_1 \mu_2 \ldots $ of generators of $G$ such that $A(\mu_j,\mu_{j+1}) = 1$ for every $j\in\mathbb N$ (where we identify $A(g_i,g_j)$ with $A(i,j)$). Let $X$ be the \emph{path space} of infinite admissible paths with the relative topology inherited as a closed and hence compact subspace of the infinite product space $\prod_{i=1}^\infty \{g_1, \ldots, g_n\}$. Let $|t|$ $(=k)$ denote the length of a reduced word $t = g_{i_1}^{\pm 1} g_{i_2}^{\pm 1} \cdots g_{i_k}^{\pm 1}$ in $G$. An action $\theta$ of $G$ on $X$ is called \emph{semisaturated} if $\theta_{ts}=\theta_t\circ \theta_s$ for every $t,s\in G$ with $|ts|=|t|+|s|$ (\ie{}when there is no reduction in the concatenation of the reduced words $t$ and $s$). By \cite{ExeLacQui}\footnote{The version on arxiv.org have an additional section on this topic. See also \cite[p.~55]{Exe}.} we have that $\ccont(X)\rtimes_{\alpha,r} G \cong \mathcal O_A$, where the partial action is the unique semisaturated partial action of $G$ on $X$ such that
\begin{align*}
X_{g_i^{-1}}&=\{\mu\in X\colon A(g_i,\mu_1)=1\} = {\rm domain( }\, \theta_{g_i}\,),\\
\theta_{g_i} &= \mu \mapsto g_i\mu,
\end{align*}
where $g_i \mu$ means concatenation of $g_i$ at the beginning of $\mu$. Using the isomorphism between $\mathcal O_A$ and $\ccont(X)\rtimes_{\alpha,r} G$,  we have that $s_i=1_{X_{g_i}}\delta_{g_i}$, where $X_{g_i}=\theta_{g_i}(X_{g_i^{-1}})$, \cf{}\cite[Theorem 7.4]{ExeLacQui} and \cite[Theorem 6.5]{Exe4}. In particular $s_is_i^*={1_{X_{g_i}}}$. 

Let $U$ be any compact and open subset of $X$. Recall that the \emph{cylinder sets} in $X$ (\ie{} the sets consisting of infinite admissible paths with the same finite initial sequence) form a basis for the topology on $X$. In particular $U$ is a finite disjoint union of cylinder sets. Since a disjoint union of compact open $(G,\tau_X)$-paradoxical sets is again $(G,\tau_X)$-paradoxical we can assume that $U$ is a cylinder set. If $U\subseteq X_{s^{-1}}$ for some $s\in G$ it follows that $U$ is $(G,\tau_X)$-paradoxical if, and only if, $\theta_s(U)\subseteq X_s$ is $(G,\tau_X)$-paradoxical. We can therefore translate $U$ until it has the form $U=X_{g_i}$ for some $i=1,\dots,n$.

We recall the definition of the $C^*$-algebra corresponding to a directed graph, \cf{}\cite{FowLacRae}. Let $E=(E^0,E^1,r,s)$ be a directed graph with countably many vertices $E^0$ and  edges $E^1$, and range and source functions $r,s:E^1\rightarrow E^0$, respectively. The $C^*$-algebra $C^*(E)$ is the universal $C^*$-algebra generated by families of projections $\{p_v:v\in E^0\}$ and partial isometries $\{s_e:e\in E^1\}$, subject to the following relations:
\begin{enumerate}[(i)] 
\item $p_vp_w=0$ for $v,w\in E^0$, $v\neq w$. 
\item $s_e^*s_f=0$ for $e,f\in E^1$, $e\neq f$. 
\item $s_e^*s_e=p_{r(e)}$ for $e\in E^1$.  
\item $s_e s_e^*\leq p_{s(e)}$ for $e\in E^1$.  
\item $p_v=\displaystyle{\sum_{\{e\in E^1:\; s(e)=v\}}}s_e s_e^*$ for $v\in E^0$ such that $0<|s^{-1}(v)|<\infty$. 
\end{enumerate}

Let $E$ be the graph corresponding to the matrix $A$, where the generators for $G$ are the vertices and where we draw an edge from $g_i$ to $g_j$ when $A(i,j)=1$. It follows form \cite[Proposition 4.1]{KunPasRae} that $C^*(E)\cong \mathcal O_A$. Using the isomorphism to identify elements in $\mathcal O_A$ and $C^*(E)$ we have that $s_i=\sum_{\{e\in E^1 \colon s(e)=g_i\}} s_e$ and $s_e=s_{s(e)}s_{r(e)}s_{r(e)}^*$, \cf{}\cite{ManRaeSut}. This gives us that
$${1_U}=1_{X_{g_i}}=s_is_i^*=\sum_{\{e\in E^1 \colon s(e)=g_i\}} s_e s_e^*=p_{g_i}.$$
Since $C^*(E)$ is purely infinite, then ${1_U}$ is properly infinite (since every nonzero positive element in a purely infinite \cst-algebra is properly infinite), but we need a bit more work to obtain the $(G,\tau_X)$-paradoxical property of $U$. 

We will now argue that $U=X_{g_i}$ is a disjoint union of sets of the form $X_{\mu_1\mu_2\dots\mu_{k+1}}$, where there is a loop based at the last vertex $\mu_{k+1}$. By a loop we mean a sequence $\nu_1\nu_2\dots\nu_{l+1}$, $l\geq 1$ of elements in $\{g_1,\dots, g_n\}$ such that $A(\nu_i,\nu_{i+1})=1$ and $\nu_{l+1}=\nu_1)$. If there is a loop based at $g_i$ we are done ($k=0$). Otherwise we consider all $g_j$ in $\{g_1,\dots, g_n\}$ such that $A(i,j)=1$. As $A$ has no zero rows we have the union $X_{g_i}=\bigcup_{\{j\,:\,A(i,j)=1\}}X_{g_ig_j}$ is non-empty. We now look at each $g_j$ to see if there is a loop based at $g_j$. If there is a loop based at $g_j$ we keep $X_{g_ig_j}$ as it is. If not, we consider all $g_k\in\{g_1,\dots, g_n\}$ such that $A(j,k)=1$ and rewrite $X_{g_ig_j}$ as $\bigcup_{\{k\,:\,A(j,k)=1\}}X_{g_ig_jg_k}$. We now look at each $g_k$ to see if there is a loop based at $g_k$. If there is a loop based at $g_k$ we keep $X_{g_ig_jg_k}$ as it is. If not, we rewrite $X_{g_ig_jg_k}$ as $\bigcup_{\{l\,:\,A(k,l)=1\}}X_{g_ig_jg_kg_l}$. We continue this process until $U$ is rewritten into the desired form. This process is finite because we pick a new element in $\{g_1,\dots, g_n\}$ every time we do a rewriting (if a vertex appears a second time when doing a rewriting then there must be a loop at that vertex, hence we never started the rewriting at that particular vertex to begin with).

Knowing that $U$ is a disjoint union of sets of the form $X_{\mu_1\mu_2\dots\mu_{k+1}}$, where there is a loop based the last vertex $\mu_{k+1}$ (and that disjoint union of paradoxical sets is paradoxical), we can assume $U=X_{\mu_1\mu_2\dots\mu_{k+1}}$, with a loop based as at $\mu_{k+1}$. With $t=(\mu_1\mu_2\dots\mu_k)^{-1}$ we have that
\begin{align*}
X_{\mu_{k+1}}&=\theta_{\mu_{k}^{-1}}(X_{\mu_{k}\mu_{k+1}})=\theta_{\mu_{k}^{-1}}(\theta_{\mu_{k-1}^{-1}}(X_{\mu_{k-1}\mu_{k}\mu_{k+1}}))\\
&=\theta_{\mu_{k}^{-1}\mu_{k-1}^{-1}}(X_{\mu_{k-1}\mu_{k}\mu_{k+1}})=\theta_{(\mu_{k-1}\mu_{k})^{-1}}(X_{\mu_{k-1}\mu_{k}\mu_{k+1}})\\
&=\theta_{(\mu_{k-2}\mu_{k-1}\mu_{k})^{-1}}(X_{\mu_{k-2}\mu_{k-1}\mu_{k}\mu_{k+1}})\\
&=\theta_{t}(X_{\mu_1\dots\mu_{k}\mu_{k+1}})=\theta_{t}(U).
\end{align*}
In particular we can assume that $U=X_{g_i}$, $g_i\in \{g_1,\dots,g_n\}$ with a loop based at $g_i$. As $C^*(E)$ is purely infinite then \cf{}\cite[Theorem 2.3]{HonSzy}, the graph $E$ satisfies condition (K) (i.e., no vertex $v\in E^0$ lies on a loop, or there are two loops $\beta',\beta''$ based at $v$ such that neither $\beta'$ nor $\beta''$ is an initial subpart of the other, \cf{}\cite{KunPasRae}). In particular there are two distinct finite loops $\beta',\beta''$ based at $g_i$. This implies that $X_{\beta'}\cup X_{\beta''}\subseteq X_{g_i}$, where the union is disjoint, and where $X_{\beta}$ denotes the set of infinite admissible path starting with the finite sequence $\beta$. If we follow the two loops $\beta'$ and $\beta''$ we come back to $g_i$, implying the existence of $t_1,t_2\in G$ such that $\theta_{t_1}(X_{\beta'})=X_{g_i}$ and $\theta_{t_2}(X_{\beta'})=X_{g_i}$. We conclude that $U=X_{g_i}$ is $(G,\tau_X)$-paradoxical.
\end{proof}	
\end{example}

\begin{example}[\cst-algebras of integral domains]
Let $R$ be an integral domain with the property that the quotient $R/(m)$ is finite, for all $m\neq 0$ in $R$.  Set $R^\times:=R\setminus\{0\}$. Following Boava and Exel \cite{BoaExe} we define the \emph{regular \cst-algebra $\mathfrak{A}[R]$ of $R$}  as the universal \cst-algebra generated by isometries $\{s_m \colon m\in R^\times\}$ and unitaries $\{u^n \colon  n\in R\}$ subject to the relations	
$$s_{m}s_{m'}=s_{mm'}, \ u^{n}u^{n'}=u^{n+n'}, \ s_{m}u^n=u^{mn}s_{m}, \ \sum_{l+(m)\in R/(m)}u^ls_m s_m^*u^{-l}=1,$$
for $m,m'\in R^\times$ and $n,n'\in R$.

Following Boava and Exel \cite{BoaExe}, let $K$ denote the field of fractions of $R$,  and $K^\times$ the set $K\backslash\{0\}$. Let $G$ be the semidirect product $K \rtimes K^\times=\{(u,w)\colon u\in K, w\in K^{\times}\}$ equipped with the following operations $(u,w)(u',w')=(u+u'w,ww')$ and $(u,w)^{-1}=(-u/w,1/w)$. As in \cite{BoaExe} we define a partial order on $K^\times$ given by $w\leq w'$ if $w'=wr$ for some $r\in R$. Let $(w)$ denote the ideal $wR\subseteq K$.

Let $X$ be the space of all sequences $(u_w+(w))_{w\in  K^\times}$ in $\prod_{w\in K^\times}(R+(w))/(w)$ fulfilling that $u_{w'}+(w)=u_w+(w)$ if $w\leq w'$. Then in  \cite{BoaExe}, Boava and Exel prove (see Appendix \ref{appendixB} for a more detailed description) that 
$$\ccont(X)\rtimes_{\alpha,r} G \cong\mathfrak{A}[R]\, ,$$ where the partial action $\theta$ of $G$ on $X$ is defined by
\begin{align*}
X_{(u,w)}&=\{(u_{w'}+(w'))_{w'}\in X \colon u_w + (w) = u + (w)\},\\
\theta_{(u,w)} &= (u_{w'}+(w'))_{w'}\mapsto (u + wu_{w^{-1}w'} + (w'))_{w'}.
\end{align*}
Note that the sets $(X_t)_{t\in G}$ are compact, open and form a basis for the topology on $X$, \cf{}\cite{BoaExe}.

\begin{proposition}
If $(\ccont(X),G,\alpha)$ denotes the partial dynamical system described in \cite{BoaExe} and recalled above, then every compact and open subset of $X$ is $(G,\tau_X)$-paradoxical. In particular $\mathcal{O}_\mathbb{Z}\cong\mathfrak{A}[\mathbb{Z}]$ is purely infinite.
\end{proposition}
\begin{proof}

Let $U$ be any compact and open subset of $X$. For $w\in K^\times$ and $C_w\subseteq (R+(w))/(w)$ set
$$V_{w}^{C_{w}}:=\{(u_{w'}+(w'))_{w'}\in X \colon u_w + (w)\in C_w\}.$$
Boava and Exel showed in \cite{BoaExe} that the family of sets $(V_{w}^{C_{w}})_{w\in  K^\times}$ is closed under complement, intersections and finite unions. It follows that $U=V_{w}^{C_{w}}$ for some $w\in K^\times$ and $C_w\subseteq (R+(w))/(w)$.  Since a disjoint union of compact open $(G,\tau_X)$-paradoxical sets is again $(G,\tau_X)$-paradoxical we can assume that $C_w$ contains precisely one element, \ie$U=X_t$ for some $t=(u,w)\in G$. Since $X_{(u,w)}=\emptyset \Leftrightarrow u\notin R+(w)$, \cf{}\cite[Proposition 4.4]{BoaExe}, we can assume that $u\in R+(w)$. Using that $X_{(u,w)}=X_{(u+w,w)}$ we can assume $u\in R$. Since $C_w\neq \emptyset$ it follows from an argument prior \cite[Proposition 4.5]{BoaExe} that $V_{w}^{C_{w}}=V_{wr}^{C_{wr}}$, with $C_{wr}$ containing at last two elements for some\footnote{In \cite{BoaExe} it is stated that any non-invertible $r$ in $R$ will work, which is not true.} $r$ in $R$. In particular
$$U= \bigcup_{s+(rw)\in C_{rw}} X_{(s,rw)}, \ \ \ |C_{rw}|>1.$$
Using the relations (with $s\in  R^\times$)
\begin{align*}
	 X_{(0,s)^{-1}}&=X & \theta_{(0,s)}(X_{(u,w)})&=X_{(su,sw)},\\
	 X_{(s,1)^{-1}}&=X & \theta_{(s,1)}(X_{(u,w)})&=X_{(s+u,w)}.
\end{align*}
it follows that there exist a finite number of elements $t_1,\dots, t_n$ in $G$ and a finite number of open, pairwise disjoint subsets $U_1, \dots, U_n$ of $U$ such that
$$U= \bigcup_{i\in \{1,\dots,n\}}U_n, \ \ \ |n|>1,\ \ \ \theta_{t_j}(U)=U_j, \ \ \ j\in \{1,\dots , n\}.$$
This shows that $U$ is $(G,\tau_X)$-paradoxical. Since $G$ is solvable both the group $G$ and the action $\alpha$ is exact. It was shown in \cite[Proposition 4.5]{BoaExe} that the action is {residually topologically} free. By Theorem \ref{the4.2} $\mathcal{O}_\mathbb{Z}\cong\mathfrak{A}[\mathbb{Z}]$ is purely infinite.
\end{proof}

\begin{corollary}
Let $(\ccont(X),G,\alpha)$ be the partial dynamical system as above with $R = \mathbb{Z}$ and $G=\QQ \rtimes \QQ^\times$.  Then 
every compact and open subset of $X$ is $(H,\tau_X)$-paradoxical, with $H=\QQ\rtimes \QQ^\times_+$. In particular $\mathcal{O}_\mathbb{N}\cong \ccont(X)\rtimes_{\alpha,r} H$ is purely infinite.
\end{corollary}
\end{example}

\section{Connection to crossed products}
\label{section5}
Abadie showed in \cite{Aba} that certain partial crossed products are Morita-Rieffel equivalent to ordinary crossed products. Since pure infiniteness is preserved under stable isomorphism (see \cite{KirRor}) one might also expect that the mentioned Morita-Rieffel equivalence maps paradoxical sets in the realm of partial actions into paradoxical sets in the realm of ordinary action. This is precisely the case.

Let us first recall the result of Abadie in \cite{Aba}. 

\begin{definition}
Let $(\czero(X),G,\alpha)$ be a partial dynamical system, together with the corresponding collection of homeomorphisms $(\newfunk{\theta_t}{X_{t^{-1}}}{X_t})_{t \in G}$, 
inducing $\alpha$. The envelope space, denoted by $X^e$, is the topological quotient space $\frac{(G \times X)}{\sim}$, where $\sim$ is the equivalence relation given by 
$$(r,x)\sim (s,y) \Leftrightarrow x\in X_{r^{-1}s} \text{ and } \theta_{s^{-1}r}(x)=y.$$ 
The {\em envelope action}, denoted by $h^e$, is the action induced in $X^e$ by 
the action $h^e_s(t,x)\mapsto(st,x)$.
\end{definition}

\begin{theorem}[Abadie \cite{Aba}]%
Let $(\czero(X),G,\alpha)$ be a partial dynamical system such that $X^e$ is Hausdorff. Let $\alpha_e$ denote the action of $G$ on $\czero(X^e)$ induced by the envelope action. Then $\czero(X)\rtimes_{\alpha,r} G$ is Morita-Rieffel equivalent to $\czero(X_e)\rtimes_{\alpha_e,r} G$.
\end{theorem}

We then have:

\begin{theorem}
\label{thed1}	
Let $(\czero(X),G,\alpha)$ be a partial dynamical system such that $X^e$ Hausdorff. If every compact open subset of $X$ is $(G,\tau_X)$-paradoxical, then every compact open subset of $X^e$ is $(G,\tau_{X^e})$-paradoxical.
\end{theorem}

\begin{proof}
Let $U^e$ be a compact open subset of $X^e$. Find a cover of $U^e$ consisting of open subsets $U^e$ of the form $\frac{\{t\}\times U_t}{\sim}$, with $U_t\in \tau_X$ and $t\in G$.	By compactness we can assume the cover is finite. Moreover, we can assume that

{(i) The union is disjoint:} Let $\rho$ denote the canonical surjection $G \times X\mapsto X^e$, and let $X_t^c$, $t\in G$, denote the complement of $X_t$ in $X$. Fix $t\in G$. For each $s\in G$ and $x\in X_{t^{-1}s}$ we have that $(t,x)\sim (s,\theta_{s^{-1}t}(x))$. In particular $\rho^{-1}(\frac{\{t\}\times X}{\sim})=\bigcup_{s\in G}\{s\}\times \theta_{s^{-1}t}(X_{t^{-1}s})$ is open in $X$, implying that $\frac{\{t\}\times X}{\sim}$ is open in $X^e$. Since $X^e$ is Hausdorff the sets $(X_t)_{t\in G}$ are clopen in $X$, \cf{}\cite[Proposition 3.1]{ExeGioGon}. Hence $\bigcup_{s\in G}\{s\}\times X_{s^{-1}t}^c$ is open in $X$, and $\frac{\{t\}\times X}{\sim}$ is therefore clopen in $X^e$. Since $X^e=\bigcup_{t\in G}\frac{\{t\}\times X}{\sim}$ se can, by compactness, cover $U^e$ by finitely many clopen set of the form $(\frac{\{t\}\times X}{\sim})_{t\in F}$. By passing to subsets of the sets $(\frac{\{t\}\times X}{\sim})_{t\in F}$ we obtain a partition $(V_t^e)_{t\in F}$ of $\frac{F\times X}{\sim}$ of clopen subsets of $X_e$ fulfilling that $V_t^e\subseteq \frac{\{t\}\times X}{\sim}$. With $U_t^e:=V_t^e\cap U^e$ we have that $U^e=\bigcup_{t\in F}U_t^e$, where the union is disjoint. For each $f\in F$ define $U_t:=\{x\in X\colon \rho_t(x)\in U_t^e\}\in \tau_X$, where $\newfunk{\rho_t}{X}{X^e}$ denotes the (continues) composition of $x\mapsto (t,x)$ and $\rho$. It follows easily that $U_t^e=\frac{\{t\}\times U_t}{\sim}$.

{(ii) Each set $U_t$ is compact:} Fix $t\in F$. Let $(V_i)_{i}$ be a cover of $U_t$ consisting of open subsets of $X$. Notice that each $V_i^e:=\frac{\{t\}\times V_i}{\sim}$ is open in $X^e$ (since $\bigcup_{s\in G}\{s\}\times \theta_{s^{-1}t}(V_i\cap X_{t^{-1}s})$ is open in $X$). The sets $(V_i^e\cap V_t^e)_{i}$ is a cover of $U_t^e$ consisting of open subsets of $X^e$ and, together with $(V_s^e)_{s\in F\setminus \{t\}}$, a cover of $U^e$. By compactness we can assume $i=1,\dots,m$ for some $m\in \NN$. Intersecting with $V_t^e$ we have that $U_t^e\subseteq \bigcup_{i=1}^m V_i^e\cap V_t^e \subseteq \bigcup_{i=1}^m V_i^e$. It follows that $U_t\subseteq \bigcup_{i=1}^m V_i$. Hence $U_t$ is compact.

{(iii) $U^e$ is paradoxical:}
Each set of the form $\frac{\{e\}\times U}{\sim}$, with $U$ compact and open in $X$, is $(G,\tau_{X^e})$-paradoxical: If $(V_i,t_i)_{i=1}^{n+m}$ witness the paradoxicality of $U$ and $k\neq l$ are any natural numbers in ${1,\dots,n+m}$ if follows that $\frac{\{e\}\times V_k}{\sim}\cap \frac{\{e\}\times V_l}{\sim}=\emptyset$ because $V_k\cap V_l=\emptyset$. We obtain that $(\frac{\{e\}\times V_i}{\sim},t_i)_{i=1}^{n+m}$ witness the paradoxicality of $\frac{\{e\}\times U}{\sim}$. A translate of a $(G,\tau_{X^e})$-paradoxical set is again $(G,\tau_{X^e})$-paradoxical. Hence each $\frac{\{t\}\times U_t}{\sim}$ is $(G,\tau_{X^e})$-paradoxical. A finite disjoint union of $(G,\tau_{X^e})$-paradoxical set is again $(G,\tau_{X^e})$-paradoxical. We conclude that $U^e$ is $(G,\tau_{X^e})$-paradoxical. 
\end{proof}

\begin{remark}
In a very interesting recent preprint \cite{KelMonRor}, Kellerhals, Monod, and R{\o}rdam proved that a countable group is non-supramenable if and only if it admits a free, minimal, purely infinite action on the locally compact non-compact Cantor set. Then they use this characterisation to associate to such dynamical systems stable Kirchberg C*-algebras in the UCT class.

Based on the examples presented in Section \ref{examples}, where
the opens sets $(X_t)_{t\in G}$ we used to construct the partial crossed products were clopen (ensuring $X^e$ is Hausdorff, see \cite[Proposition 3.1]{ExeGioGon}), and Proposition 6.3, we obtain another class of examples of dynamical systems whose associated crossed products are stable Kirchberg C*-algebras.
\end{remark}

\appendix
\section{}
\label{partail}
\subsection{Basic definitions} 
\label{appendixA1}
Recall that a \emph{Fell bundle} over a discrete group $G$ is a collection $\BB = (B_t)_{t\in G}$ of closed subspaces of a \cst-algebra $\mcl{B}$, indexed by a discrete group $G$, satisfying $B_t^* = B_t$ and $B_t B_s \subseteq B_{ts}$ for all $t$ and $s$ in $G$ (\cf\cite{Exe}). A \emph{section} of $\BB$ is a function $\newfunk{\xi}{G}{\bigcup_{t\in G}B_t}$ with the property that $\xi(t)\in B_t$ for all $t\in B_t$. Let $l_1(\BB)$ denote the Banach $^*$-algebra (\cf\cite[Proposition 2.1]{McClan}) consisting of all sections $\xi$ of $\BB$ with a finite $l_1$-norm.  The operations and the norm in $l_1(\BB)$ are given by
$$\xi\eta(t)=\sum_{s\in G} \xi(s)\eta(s^{-1}t), \ \ \ \xi^*(t)=\xi(t^{-1})^*,$$
$$\|\xi\|_1=\sum_{t\in G}\|\xi(t)\|.$$
We define the \emph{full cross sectional algebra} of $\BB$, denoted $\ccont^*(\BB)$, to be the enveloping \cst-algebra of $l_1(\BB)$ (\cf\cite[Section VIII.17.2]{FelDor}). Recall that a \emph{right Hilbert $A$-module} $X$ is a right $A$-module $X$ equipped with a map $\langle \cdot, \cdot \rangle_A : X \times X \to A$ that is linear in the second component and for $x,y \in X$, $a \in A$,
\begin{enumerate}[(i)]
\item $\langle x,x \rangle_A \geq 0$ with equality only if $x = 0$;
\item \label{A-linear in the second variable} $\langle x,y \cdot a\rangle_A=\langle x,y\rangle_Aa$;
\item\label{it:adjoint of ip}$\langle x,y\rangle_A = \langle y,x\rangle_A^*$; and
\item $X$ is complete in the norm defined by $\|x\|_A ^2 = \|\langle x, x\rangle_A\|$.
\end{enumerate}
As in \cite[Section 1.1.7]{JenTho}, let $l_2(\BB)$ denote the right Hilbert $B_e$-module consisting of all cross sections $\xi$ of $\BB$ fulfilling that the series $x=\sum_{t\in G} \xi(t)^* \xi(t)$ converges unconditionally (\ie{}for any $\varepsilon>0$ there exist a finite set $F\subseteq G$ such that for every finite subsets $H\subseteq G$ containing $F$ one has that $\|x-\sum_{t\in H} \xi(t)^* \xi(t)\|<\varepsilon$). We equip $l_2(\BB)$ with the inner product
$$
\langle\xi,\eta\rangle = \sum_{t\in G} \xi(t)^* \eta(t), \ \ \ \xi,\eta\in l_2(\BB).
$$
Let $\Ll(l_2(\BB))$ denote the \cst-algebra of all adjointable operators on $l_2(\BB)$ (\ie{}linear maps $\newfunk{T}{l_2(\BB)}{l_2(\BB)}$ with a linear map $T^*$ such that $\langle\xi,T\eta\rangle = \langle T^*\xi,\eta\rangle$, for all $\xi,\eta\in l_2(\BB)$). With the $^*$-homomorphism $\newfunk{\Lambda_\BB}{l_1(\BB)}{\Ll(l_2(\BB))}$ defined (in \cite[Proposition 2.6]{ExeNg}) by
$$
(\Lambda_\BB(\xi)\eta)(t) = \sum_{s\in G}\xi(s)\eta(s^{-1}t), \ \ \ \xi\in l_1(\BB), \ \eta\in l_2(\BB), \ t\in G,
$$
we define the \emph{reduced cross sectional algebra} of $\BB$, denoted $\ccont_r^*(\BB)$, to be the sub-\cst-algebra of $\Ll(l_2(\BB))$ generated by the range of $\Lambda_\BB$. Let $\delta_{s,t}$, $s,t\in G$ denote the Kronecker symbol. Each $B_t$ is a right Hilbert $B_e$-module with the inner product $\langle b,c \rangle = b^*c$. Let $j_t\in \Ll(B_t,l_2(\BB))$ denote the adjointable operator defined by $(j_t(b_t))(s)=\delta_{s,t}b_t$, for $b_t\in B_t$ and $s\in G$. The adjoint $j_t^*\in \Ll(l_2(\BB),B_t)$ is simply the map $j_t^*(\xi)=\xi(t)$, $\xi\in l_2(\BB)$. Recall, \cf\cite[Definition 2.7]{Exe}, that for $x\in\ccont_r^*(\BB)$ and $t\in G$ the \emph{$t^{\rm th}$ Fourier coefficient} of $x$ is the unique element $\hat{x}(t)\in B_t$ such that $(j_t^*\circ x \circ j_e)(a)=\hat{x}(t)a$ for all $a\in B_e$. The map $\newfunk{E}{\ccont_r^*(\BB)}{B_e}$, given by $x\mapsto \hat{x}(e)$ is a positive, contractive, conditional expectation. Moreover, $E$ is faithful on positive elements since $E(x^*x)=\sum_{t\in G}\hat{x}(t)^*\hat{x}(t)$ (unconditional convergence) for each $x\in \ccont_r^*(\BB)$ (\cf\cite[Proposition 2.12]{Exe}). Let $\ccomp(\BB)\subseteq l_1(\BB)$ denote the set of finitely supported sections of $\BB$. 

Let $\mcl{A}$ be a \cst-algebra and $G$ be a discrete group. Recall (see \cite{DokExe}) that a {\em partial action of $G$ on $\mcl{A}$}, denoted by $\alpha $, is a collection $(\mcl{D}_t)_{t \in G}$ of closed two-sided ideals of $\mcl{A}$ and a collection $(\alpha_t)_{t \in G}$ of $^*$-isomorphisms $\newfunk{\alpha_t}{\mcl{D}_{t^{-1}}}{\mcl{D}_t}$ such that
\begin{enumerate}[(i)]
\item $\mcl{D}_e=\mcl{A}$, where $e$ represents the identity element of $G$;
\item $\alpha_s^{-1}(\mcl{D}_s\cap \mcl{D}_{t^{-1}})\subseteq\mcl{D}_{(ts)^{-1}}$;
\item $\alpha_t\circ \alpha_s(x)=\alpha_{ts}(x), \ \ \forall \: x\in\alpha_s^{-1}(\mcl{D}_s\cap \mcl{D}_{t^{-1}})$.
\end{enumerate}
The triple $(\mcl{A},G,\alpha)$ is called a {\em partial dynamical system}. Fix a  partial dynamical system $(\mcl{A},G,\alpha)$. Recall the corresponding Fell bundle $\BB$ as defined in Definition \ref{delpartialsystem}: We let $\mcl{L}$ be the normed $^*$-algebra of the finite formal sums $\sum_{t \in G} a_t\delta_t$, where $a_t\in\mcl{D}_t$. We let $B_t$ denote the vector subspace $\mcl{D}_t\delta_t$ of $\mcl{L}$. The family $(B_t)_{t \in G}$ generates the Fell bundle $\BB$. It follows that $\ccomp(\BB)=\mcl{L}$. Moreover, for $x=\sum_{t\in G}a_t\delta_t\in \ccomp(\BB)$, we have that $E(\Lambda_\BB(x))={a_e}$. 
\begin{lemma}[McClanahan]
\label{lemmaA1}
Let $(\mcl{A},G,\alpha)$ be a partial dynamical system. For any $G$-invariant ideal $\mcl{I}$ of $\mcl{A}$ we have a canonical embedding of $\mcl{I}\rtimes_{\alpha,r}G$ in $\mcl{A}\rtimes_{\alpha,r}G$.
\end{lemma}
\begin{remark}
To best of {our} knowledge our proof of Lemma \ref{lemmaA1} is new. For a different proof using covariant representation we refer to the work in \cite{McClan}. Since our proof applies to general Fell bundles (and not only the one defining crossed products) we have included it for completeness.
\end{remark}	
\begin{proof}
Fix an approximate unit $(e_n)$ for $\mcl{I}$, and the Fell bundle 
$$\EE=(E_t)_{t\in G}, \ \ \ E_t=(\mcl{D}_t\cap\mcl{I})\delta_t,$$ 
equipped with the operations and norm coming from $\BB$. For each section $\xi\in l_2(\BB)$ let $\newfunk{\xi_n}{G}{\bigcup_{t\in G}E_t}$ denote the map $t\mapsto \xi(t){e_n}$ contained in $l_2(\EE)$. Let $\newfunk{\varphi}{\Ll(l_2(\EE))}{\Ll(l_2(\BB))}$ denote the map ${\varphi(T)\xi=\lim_n T(\xi_n)}$. For $x\in \ccomp(\EE)$, $\varphi(\Lambda_\EE(x))=\Lambda_\BB(x)$.
We have that $\varphi(\Lambda_\EE(x))=\Lambda_\BB(x)$ and $\|\varphi(\Lambda_\EE(x))\|=\|\Lambda_\EE(x)\|$, for any element $x=\sum_{t\in G}a_t\delta_t$ in $\ccomp(\EE)$, by direct computation
\begin{align*}
	\|\varphi(\Lambda_\EE(x))\|^2&=\sup_{\xi\in \ccomp(\BB), \|\xi\|\leq 1}\|\varphi(\Lambda_\EE(x))\xi\|^2\\
	&=\sup_{\xi\in \ccomp(\BB), \|\xi\|=1}\lim_n\|\sum_{t,s,r\in G}\big((a_t\delta_t)\xi_n(t^{-1}r)\big)^*((a_s\delta_s)\xi_n(s^{-1}r))\|\\
	&=\sup_{\eta\in \ccomp(\EE), \|\eta\|\leq 1}\|\sum_{t,s,r\in G}\big((a_t\delta_t)\eta(t^{-1}r)\big)^*((a_s\delta_s)\eta(s^{-1}r))\|\\
	&=\sup_{\eta\in \ccomp(\EE), \|\eta\|\leq 1}\|\Lambda_\EE(x)\eta\|^2\\
	&=\|\Lambda_\EE(x)\|^2
\end{align*}
The first and last equality above use the fact that the finitely supported sections are dense in $l^2$. The second and fourth equality follows from the definition of the maps $\varphi$ and $\Lambda_\EE$. For the third equality we obviously have $\geq$ after taking the limit. To get the $\leq$ notice that if we remove ``$\sup_{\xi\in \ccomp(\BB), \|\xi\|\leq1}\lim_n$'' from the left hand side we have $\leq$ since $\xi_n\in \{\eta\in \ccomp(\EE), \|\eta\|\leq 1\}$ for any $n$ and any $\xi\in \ccomp(\BB)$ with $\|\xi\|\leq 1$. The inequality $\leq$ remains valid when we take the limit and the supremum. I particular we obtain that the canonical embedding of $\ccomp(\EE)$ into $\ccomp(\BB)$, given by $a_t\delta_t\mapsto a_t\delta_t$, extends to an embedding $\newfunk{\iota}{\mcl{I}\rtimes_{\alpha,r}G}{\mcl{A}\rtimes_{\alpha,r}G}$. 
\end{proof}

Let $\mcl{I}$ be a $G$-invariant ideal of $\mcl{A}$. In order to introduce the map $\newfunk{\rho}{\mcl{A}\rtimes_{\alpha,r}G}{\mcl{A/I}\rtimes_{\alpha,r}G}$ we can again turn to \cite{Exe}. Let $\FF$, denote the Fell bundle 
$$(F_t)_{t\in G}, \ \ \ F_t=(\mcl{D}_t/\mcl{(I\cap D}_t))\delta_t,$$ 
equipped with the operations and norm coming from $\BB$. The canonical surjection of $\ccomp(\BB)$ onto $\ccomp(\FF)$, given by $a_t\delta_t\mapsto (a_t+ \mcl{D}_t\cap \mcl{I})\delta_t$, extends to a surjective $^*$-homomorphism $\newfunk{\rho}{\mcl{A}\rtimes_{\alpha,r}G}{\mcl{A/I}\rtimes_{\alpha,r}G}$ (\cf\cite[Proposition 3.11]{Exe}). In particular, by continuity of $\iota$, $\rho$, and $E$, we have the following result:
\begin{proposition}[Exel, McClanahan]
\label{appendix1}
Let $(\mcl{A},G,\alpha)$ be a partial dynamical system. For any $G$-invariant ideal $\mcl{I}$ of $\mcl{A}$ we have the commuting diagram
$$
\label{diagram}
\xymatrix{0 \ar[r] & \mcl{I}\rtimes_{\alpha,r}G \ar[r]^-{\iota} \ar[d]^-{E_\mcl{I}} & \mcl{A}\rtimes_{\alpha,r}G  \ar[r]^-{\rho} \ar[d]^-{E_\mcl{A}}& \mcl{A/I}\rtimes_{\alpha,r}G \ar[r] \ar[d]^-{E_{\mcl{A/I}}}& 0\\
0 \ar[r] & \mcl{I} \ar[r] & \mcl{A} \ar[r] & \mcl{A/I} \ar[r] & 0}
$$
\end{proposition}
Using that $E_{\mcl{A}}$ is idempotent and identifying $\mcl{I}\rtimes_{\alpha,r}G$ with a subset of $\mcl{A}\rtimes_{\alpha,r}G$ we have the additional properties:
\begin{proposition}[Exel, McClanahan]
\label{appendix3}
Let $(\mcl{A},G,\alpha)$ be a partial dynamical system. For any $G$-invariant ideal $\mcl{I}$ of $\mcl{A}$ and any ideal $\mcl{J}$ of in $\mcl{A}\rtimes_{\alpha,r}G$ we have the following identities:
$$(\mcl{I}\rtimes_{\alpha,r}G)\cap \mcl{A} = \mcl{\mcl{I}}, \ \ \ \mcl{J\cap A}\subseteq E_{\mcl{A}}(\mcl{J}), \ \ \ {\rm Ideal}[\mcl{I}]=\mcl{I}\rtimes_r G,$$
where ${\rm Ideal}[S]$ denotes the smallest ideal in $\mcl{A}\rtimes_{\alpha,r}G$ generated by $S\subseteq \mcl{A}\rtimes_{\alpha,r}G$.
\end{proposition}
\subsection{Topological freeness}
\label{appendixA2}
Let $(\mcl{A},G,\alpha)$ be a partial dynamical system with ideals $(\mcl{D}_t)_{t \in G}$ and $^*$-isomorphisms $(\alpha_t)_{t \in G}$. If we replace the \cst-algebra $\mcl{A}$ by a locally compact Hausdorff space $X$, the ideals $\mcl{D}_t$ by open sets $X_t$ and the $^*$-isomorphisms $\alpha_t$ by homeomorphisms  $\newfunk{\theta_t}{X_{t^{-1}}}{X_t}$, we obtain a {\em partial action $\theta$} of $G$ on the space $X$. A partial action of a group $G$ on a space $X$ induces naturally a partial action $\alpha$ of $G$ on $\czero(X)$. The ideals are $\czero(X_t)$ and $\alpha_t(f)=f\circ \theta_{t^{-1}}$. The converse (still in the abelian case) is also true (\cf\cite{Aba, Gon}).

Let ${\rm Prim}\,\mcl{A}$ denote the primitive ideal space of $\mcl{A}$  with respect to the Jacobson topology. Let $\hat{\mathcal{A}}$ denote the spectrum of $\mcl{A}$  with respect to the topology induced by the surjection $\newfunk{\kappa}{\hat{\mathcal{A}}}{{\rm Prim}\,\mcl{A}}$, $\kappa([\pi])={\rm ker }\,\pi$. Following \cite{Leb} recall how $\alpha$ defines a partial action of $G$ on ${\rm Prim}\,\mcl{A}$ and on $\hat{\mathcal{A}}$:

For any ideal $\mcl{J}$ of $\mcl{A}$ we let ${\rm supp}\, \mcl{J}$ denote the subset $\{x\in {\rm Prim}\,\mcl{A}\colon \mcl{J}\nsubseteq x \}$. It is known (see \cite[Section 3.2.1]{Dix} or \cite[Section 1.4]{Sie2}) that the  mapping $x \mapsto x\cap J$ establishes a homeomorphism ${\rm supp}\, \mcl{J} \leftrightarrow {\rm Prim}\,\mcl{J}$ and ${\rm supp}\, \mcl{J}$ is an open set in ${\rm Prim}\,\mcl{A}$. Set also $\hat{\mathcal{A}}^\mcl{J} = \{ [\pi] \in \hat{\mathcal{A}} : \pi (\mcl{J}) \neq 0 \}$. Then the mapping $[\pi] \mapsto [\pi|_\mcl{J}]$ establishes a homeomorphism $\hat{\mathcal{A}}^\mcl{J}\leftrightarrow \hat{\mathcal{J}}$ and $\hat{\mathcal{A}}^\mcl{J}$ is an open set in $\hat{\mathcal{A}}$  (see \cite[Section 3.2.1]{Dix} or \cite[Section 1.4]{Sie2}). Let us define the mapping $\newfunk{\tau_t}{\hat{\mathcal{A}}^{\mcl{D}_{t^{-1}}}}{\hat{\mathcal{A}}^{\mcl{D}_{t}}}$ in the following way: For any $[\pi] \in \hat{\mathcal{A}}^{\mcl{D}_{t^{-1}}}$ we set 
$$\tau_t ([\pi]) = [\pi\circ\alpha_{t^{-1}}], \ \ \ t\in G.$$ 
The foregoing observations tell us that  $\tau_t$ is a homeomorphism. Let us also define the mapping $\newfunk{\theta_t}{{\rm supp}\, \mcl{D}_{t^{-1}}}{{\rm supp}\, \mcl{D}_t}$ in the following way: For any point $x \in {\rm supp}\, \mcl{D}_{t^{-1}}$ such that $x = {\rm ker}\, \pi$ where $[\pi] \in \hat{\mathcal{A}}^{\mcl{D}_{t^{-1}}}$ we set 
$$
\theta_t(x) = \ker\, (\pi\circ\alpha_{t^{-1}}), \ \ \ t\in G.$$
Clearly $\theta_t$ is a homeomorphism. For $\tau_t$ and $\theta_t$ defined in the above described way we have that $(\tau_t)_{t\in G}$ defines a partial action of $G$ on $\hat{\mathcal{A}}$ and $(\theta_t)_{t\in G}$  defines a partial action of $G$ of ${\rm Prim}\,\mcl{A}$. Recall that the action $\alpha$ is called \emph{topologically free} if for any finite set $\{t_1,\dots, t_n\}$, $t_i \neq e$ the set 
$$\bigcup_{i=1}^n \big\{x \in {\rm supp}\, \mcl{D}_{t_i^{-1}}: \theta_{t_i} (x)=x   \big\}$$ 
has empty interior. This condition can be also formulated in the following way: For any finite set $\{ t_1 , ... t_k \} \subseteq G$ and any nonempty open set $U$ there exists a point $x \in U $ such that all the points $\theta_{t_i}(x)$, $i=1,\dots,k$  that are defined  ($ \Leftrightarrow x \in {\rm supp}\, \mcl{D}_{t_i^{-1}}$) are distinct. Recall that the action $\alpha$ has the \emph{intersection property} if every nonzero ideal in $\mcl{A}\rtimes_{\alpha,r} G$ has a nonzero intersection with $\mcl{A}$.
\begin{theorem}[Lebedev]
\label{appendix2}
Let $(\mcl{A},G,\alpha)$ be a partial dynamical system. Suppose that the action is topologically free. Then $\alpha$ has the intersection property.
\end{theorem}

\begin{theorem}[Exel, Laca, Quigg]
\label{appendix4}
Let $(\mcl{A},G,\alpha)$ be a partial dynamical system with $\mcl{A}$ abelian. Suppose that the action is topologically free. Then for every $b\in \mcl{A}\rtimes_{\alpha,r} G$ and every $\varepsilon>0$ there exist a positive contraction $x\in \mcl{A}$ satisfying 
$$\|xE_{\mcl{A}}(b)x-xbx\|< \varepsilon, \ \ \ \|xE_{\mcl{A}}(b)x\|> \|E_{\mcl{A}}(b)\|-\varepsilon.$$
\end{theorem}

\begin{proposition}
\label{abeliantopfree2}
Let $(\mcl{A},G,\alpha)$ be a partial dynamical system with $\mcl{A}$ abelian. Then the following properties are equivalent:
\begin{enumerate}[(i)]
\item $\alpha$ is topologically free,
\item[$(i')$] $F_t=\{x\in U_{t^{-1}}: \theta_t(x)= x\}$ has empty interior ($t\neq e$),
\item $\|\alpha_t|_{\mcl{K}}-id|_\mcl{K}\|=2$ for every $\alpha_t$-invariant ideal $\mcl{K}$ in $D_{t^{-1}}$ (and $t\neq e$),
\item[$(ii')$] $\|\alpha_t|_{\mcl{K}}-id|_\mcl{K}\|\neq 0$ for every $\alpha_t$-invariant ideal $\mcl{K}$ in $D_{t^{-1}}$ (and $t\neq e$),
\item $\inf\{\|x\alpha_{t}(x)\|: x\in \mcl{K}_+, \|x\|=1\}=0$ for every nonzero ideal $\mcl{K}$ in $\mcl{D}_{t^{-1}}$ (and $t\neq e$),
\item $\inf\{\|{x}( b_t\delta_t){x}\|: x\in \mcl{B}_+, \|x\|=1\}=0$ for every $b_t\in{\mcl{D}_t}$, and every nonzero hereditary \cst-algebra $\mcl{B}$ in $\mcl{A}$ (and $t\neq e$),
\item[$(iv')$] $\inf\{\|{x}( b_t\delta_t){x}\|: x\in \mcl{K}_+, \|x\|=1\}=0$ for every $b_t\in{\mcl{D}_t}$, and every nonzero ideal $\mcl{K}$ in $\mcl{A}$ ($t\neq e$).
\end{enumerate}
\end{proposition}

\begin{proof} 
The implications $(i)\Leftrightarrow (i')$ follows from \cite[Definition 2.1, Lemma 2.2]{ExeLacQui}, and the implications $(iv)\Leftrightarrow (iv')$ and $(ii)\Rightarrow (ii')$ are evident.
	
$(ii')\Rightarrow (ii)$: 
Find some function $f\in \mcl{K_+}=\czero(V)_+$ such that $\alpha_t(f)\neq f$. Pick any $x\in C$, where $C:=\{y\in V \colon |\alpha_t(f)(y)-f(y)|>0\}$.

Notice $x\neq \theta_{t^{-1}}(x)$ in $V$. \big(If $f(x)\neq 0$ then---for $h:=\alpha_{t^{-1}}(f)\in \mcl{K}$---we have that $h(\theta_{t^{-1}}(x))=f(x)\neq 0$, recalling that $\alpha_t(f)(x)=f(\theta_{t^{-1}}(x))$. Hence $\theta_{t^{-1}}(x)\in V$. If $f(x)=0$ then $\alpha_t(f)(x)=f(\theta_{t^{-1}}(x))\neq 0$. Hence $\theta_{t^{-1}}(x)\in V$. If $x=\theta_{t^{-1}}(x)$ then $0=|f(\theta_{t^{-1}}(x))-f(x)|=|\alpha_t(f)(x)-f(x)|>0$.\big) One can now easily find a function $h\in \mcl{K}$ such that $|h(\theta_{t^{-1}}(x))-h(x)|=2$. Hence $\|\alpha_t|_{\mcl{K}}-id|_\mcl{K}\|=2$.

$(i')\Rightarrow (iii)$:
Assume not $(iii)$. Find some $t\neq e$ and $\mcl{K}=\czero(V)$ in $\czero(U_{t^{-1}})$ such that 
$\inf\{\|x\beta(x)\|: x\in \mcl{K}_+, \|x\|=1\}>0$,
where $\beta=\alpha_t$. Suppose $(i')$. Then the $U_{t^{-1}}$-open set $F_t^c=\{ x\in U_{t^{-1}}: \theta_t(x)\neq x\}$ is dense in $U_{t^{-1}}$. As $V$ is $U_{t^{-1}}$-open it has a non-empty intersection with $F_t^c$. Therefore we can find a non-empty $U_{t^{-1}}$-open set $Y$ in $V$ such that $\theta_t(Y)$ is disjoint from $Y$. Hence there is a function $x$ in $\mcl{K}_+$ of norm one such that $x\beta(x)=0$. Contradiction, hence not $(i')$.

$(iii)\Rightarrow (ii')$:
Assume not $(ii')$. Then there exists a nonzero $\alpha_t$-invariant ideal $\mcl{K}=\czero(V)$ in $\czero(U_{t^{-1}})$ (for some $t\neq e$) such that $\|\alpha_t|_\mcl{K} - id|_\mcl{K}\|=0$. We conclude that $\alpha_t(x)=x$ for all $x \in \mcl{K}$. Hence $\inf\{\|x\alpha_t(x)\|: x\in \mcl{K}_+, \|x\|=1\}$ is nonzero implying not $(iii)$.

$(ii')\Rightarrow (i')$:
Assume not $(i')$. Find $t\neq e$ and a $U_{t^{-1}}$-open nonempty subset $V$ in $F_t$ (recall that $F_t$ is $U_{t^{-1}}$-closed but has nonempty interior). Since $\theta_t(U_{t^{-1}})=U_{t}$ we see that each $x\in V\subseteq U_{t^{-1}}$ also belongs to $U_{t}$ and $\theta_{t^{-1}}(x)=x$, $x\in V$. Hence $\mcl{K}:=\czero(V)$ is $\alpha_t$-invariant (\ie{}$\alpha_t(f)=f=\alpha_{t^{-1}}(f)$, $f\in \mcl{K}$). We obtain that $\|\alpha_t|_\mcl{K} - id|_\mcl{K}\|=0$. Hence not $(ii')$.

$(iii)\Rightarrow (iv')$: 
Fix any $t\neq e$, any $b_t\in \mcl{D}_t$, and any nonzero ideal $\mcl{K}$ in $\mcl{A}$. Define $\mcl{I}:=\mcl{K}\cap \mcl{D}_{t^{-1}}$. Suppose $\mcl{I}=0$. Fix any $x\in \mcl{K}$. Using that $xb_t\in \mcl{D}_t$ we get $\alpha_{t^{-1}}(xb_t)\in \mcl{D}_{t^{-1}}$ and $\alpha_{t^{-1}}(xb_t)x=0$. Hence ${x}( b_t\delta_t){x}=\alpha_t(\alpha_{t^{-1}}(xb_t)x)\delta_t=0$ and
$$\inf\{\|{x} b_t\delta_t){x}\|: x\in \mcl{K}_+, \|x\|=1\}=0.$$
Suppose $\mcl{I}\neq 0$. Then $\mcl{I}$ is an ideal in $\mcl{D}_{t^{-1}}$ and
$$\|{x}( b_t\delta_t){x}\|\leq \|\alpha_t(\alpha_{t^{-1}}(xb_t))\alpha_t(x)\|\leq \|b_t\|\|x\alpha_t(x)\|,\ \ \  x\in \mcl{I}_+.$$
It follows from $(iii)$ that
$$\inf\{\|{x}( b_t\delta_t){x}\|: x\in \mcl{K}_+, \|x\|=1\}=0.$$

$(iv')\Rightarrow (iii)$: 
Fix any $t\neq e$ and nonzero ideal $\mcl{K}$ in $\mcl{D}_{t^{-1}}$. Define $\mcl{I}:=\mcl{D}_{t}\cap \mcl{K}$.
Suppose $\mcl{I}=0$. Fix any $x\in \mcl{K}$. Then $x\in \mcl{D}_{t^{-1}}$, $\alpha_{t}(x)\in \mcl{D}_{t}$, $\alpha_{t}(x)x\in \mcl{D}_{t}\cap \mcl{K}$, and $\alpha_{t}(x)x=0$. Hence
$$\inf\{ \|x\alpha_t(x)\|: x\in \mcl{K}_+, \|x\|=1\|\}=0.$$
Suppose $\mcl{I}\neq 0$. Then $\mcl{I}:=\czero(V)$ is a nonzero ideal in $\mcl{A}$. Pick any $x\in V$. Choose an open set $U$ with compact closure such that 
$$x\in U\subseteq \bar{U}\subseteq V.$$
With $K:=\bar{U}\subseteq V$ there exist by Urysohn's Lemma a function $h\in\czero(V)$ such that
$$0\leq h\leq 1, \ \ \ h|_K=1.$$
Using $(iv)$ on $b_t:=h\in \mcl{D}_t$ and the nonzero ideal $\mcl{J}:=\czero(U)$ in $\mcl{A}$ (and the equality ${x}( b_t\delta_t){x}=xb_t\alpha_t(x)\delta_t$ valid for every $x\in \mcl{J}\subseteq \mcl{I} \subseteq \mcl{K}\subseteq \mcl{D}_{t^{-1}}$) we obtain
$$\inf\{ \|xb_t\alpha_t(x)\|: x\in \mcl{J}_+, \|x\|=1\|\}=0.$$
Since $h|_U=1$ we have that $xb_t=x$ for every $x\in J_+$, hence
$$\inf\{ \|x\alpha_t(x)\|: x\in \mcl{J}_+, \|x\|=1\|\}=0.$$
Since $\mcl{J}\subseteq \mcl{I} \subseteq \mcl{K}$ also
$$\inf\{ \|x\alpha_t(x)\|: x\in \mcl{K}_+, \|x\|=1\|\}=0.$$
\end{proof}

\subsection{\cst-algebras of an Integral Domain}
\label{appendixB}

Let $R$ be an integral domain (\ie{}a commutative unital ring without zero divisor). Set $R^\times:=R\setminus\{0\}$. Following Boava and Exel \cite{BoaExe} we impose that the quotient $R/(m)$ is finite, for all $m\neq 0$ in $R$ and let $\mathfrak{A}[R]$ denote the \emph{regular \cst-algebra of $R$}, i.e., the universal \cst-algebra generated by isometries $\{s_m \colon m\in R^\times\}$ and unitaries $\{u^n \colon  n\in R\}$ subject to the relations	
$$s_ms_{m'}=s_{mm'}, \ u^nu^{n'}=u^{n+n'}, \ s_mu^n=u^{mn}s_m, \ \sum_{l+(m)\in R/(m)}u^ls_m s_m^*u^{-l}=1,$$
for every $m,m'\in R^\times$ and $n,n'\in R$.
As in \cite{BoaExe} we let $K$ denote the field of fractions of $R$ and let $K^\times$ denote the set $K\backslash\{0\}$. To form a group one equip the semidirect product $K \rtimes K^\times=\{(u,w)\colon u\in K, w\in K^{\times}\}$ with the following two group operations 
$$(u,w)(u',w')=(u+u'w,ww'), \ \ \ (u,w)^{-1}=(-u/w,1/w).$$
Boava and Exel showed in \cite{BoaExe} that the regular \cst-algebra of $R$ is a partial crossed product by first showing that the algebra $\mathfrak{A}[R]$ is $^*$-isomorphic to a \emph{partial group algebra of $K \rtimes K^\times$} (with certain relations \mcl{R}), and then use that every partial group algebra is a partial crossed product.
We will not describe these isomorphisms but will instead focus on the the description of the partial crossed product.

We have a partial order on $K^\times$ given by $w\leq w'$ if $w'=wr$ for some $r\in R$. For $w\in K$ let $(w)$ denote the ideal $wR\subseteq K$. For each pair $w\leq w'$ in $K^\times$ we have a canonical map $\newfunk{p_{w,w'}}{(R+(w'))/(w')}{(R+(w))/(w)}$ given by $$p_{w,w'}(u_{w'}+(w'))=u_{w'}+(w).$$
Boava and Exel proved that the inverse limit $\lim_{\longleftarrow}\{(R+(w))/(w), \ p_{w,w'}\}$ is isomorphic to the space $X$ of all sequences 
$$(u_w+(w))_{w\in  K^\times} \in \prod_{w\in K^\times}(R+(w))/(w)$$ 
fulfilling that $u_{w'}+(w)=u_w+(w)$ if $w\leq w'$. When $(R+(w))/(w)$ is given the discrete topology and $(R+(w))/(w)$ the product topology $X$ becomes a compact topological space. Moreover, \cf{}\cite{BoaExe}, there is a partial action $\theta$ of $G$ on $X$ defined by
\begin{align*}
X_{(u,w)}&=\{(u_{w'}+(w'))_{w'}\in X \colon u_w + (w) = u + (w)\},\\
\theta_{(u,w)} &= (u_{w'}+(w'))_{w'}\mapsto (u + wu_{w^{-1}w'} + (w'))_{w'},
\end{align*}
for any $(u,w) \in K\rtimes K^\times$. The partial actions $\theta$ induces a partial dynamical system $(\ccont(X), K \rtimes K^\times, \alpha)$, and hence also a partial crossed product $\ccont(X)\rtimes_{\alpha,r} (K \rtimes K^\times)$.
\begin{theorem}[Boava-Exel]
\label{appendixB1}
Let $R$ be a integral domain with finite quotients $R/(m)$, $m\in R^\times$. Suppose that $R$ is not a filed. Then the maps
$$u_n\mapsto 1_X\delta_{(n,1)}, \ \ \ s_m\mapsto 1_{X_{(0,m)}}\delta_{(0,m)}$$
induce a $^*$-isomorphism between $\mathfrak{A}[R]$ and $\ccont(X)\rtimes_{\alpha,r} (K \rtimes K^\times)$.
\end{theorem}

\bibliographystyle{amsplain}

\providecommand{\bysame}{\leavevmode\hbox to3em{\hrulefill}\thinspace}
\providecommand{\MR}{\relax\ifhmode\unskip\space\fi MR }
\providecommand{\MRhref}[2]{%
  \href{http://www.ams.org/mathscinet-getitem?mr=#1}{#2}
}
\providecommand{\href}[2]{#2}

\end{document}